\documentclass{article}

\usepackage[utf8]{inputenc} %
\usepackage[T1]{fontenc}    %
\usepackage{hyperref}       %
\usepackage{url}            %
\usepackage{booktabs}       %
\usepackage{amsfonts}       %
\usepackage{microtype}      %

\usepackage[usenames,dvipsnames]{xcolor} %
\usepackage[activeacute,activegrave]{babel}
\usepackage[tbtags,intlimits]{amsmath}
\usepackage{physics}
\usepackage{amsthm,amssymb}
\usepackage{array}
\usepackage{mathtools}  %
\usepackage{placeins} %
\usepackage[space,multidot]{grffile} %
\usepackage{geometry}
\usepackage{qtree}
\usepackage{tikz}
\usepackage{listings}
\usepackage[colorinlistoftodos,prependcaption,textsize=tiny]{todonotes}
\usepackage{algorithm}
\usepackage{algpseudocode}
\usepackage{xfp}
\usepackage{caption}
\usepackage{subcaption}

\makeatletter

\usepackage{graphicx}

\makeatletter

\usepackage[all]{hypcap}

\newtheorem{theorema}{Theorem}

\theoremstyle{definition}

\newtheorem{consequence}{Consequence}

\allowdisplaybreaks
\def\abs|#1|{\ensuremath{\left\lvert#1\right\rvert}}
\def\norm|#1|{\ensuremath{\left\lVert#1\right\rVert}}
\DeclareMathOperator{\dom}{dom}
\DeclareMathOperator*{\argmax}{arg\,max}

\DeclareMathOperator*{\maximize}{maximize}
\mathtoolsset{showonlyrefs}

\DeclareRobustCommand\onedot{\futurelet\@let@token\@onedot}
\def\@onedot{\ifx\@let@token.\else.\null\fi\xspace}

\def\ie{\emph{i.e}\onedot}

\makeatother

\def\range#1#2{\ensuremath{\overline{#1..#2}}}
\def\range#1#2{\ensuremath{#1,\,\ldots,\,#2}}
\def\rangeo#1{\range1{#1}}
\def\rangeidx#1#2{\ensuremath{ {#1}_1,\,{#1}_2,\,\ldots,\,{#1}_{#2}  }}
\def\rangeidxO#1#2{\ensuremath{ {#1}_0,\,{#1}_1,\,\ldots,\,{#1}_{#2}  }}
\def\ie{i.\,e.}
\def\defval{\ensuremath{:=}}
\def\N_#1{\ensuremath{N_{#1}}}
\def\AlgComm#1{\# \textit{#1}}

\def\fnum#1{\ensuremath{f^{(#1)}}}
\def\f^#1_#2(#3){\ensuremath{f^{(#1)}_{#2}(#3)}}
\def\f^#1_#2(#3){\ensuremath{f^{(#1)}(#2,\,#3)}}
\def\tc#1(#2,#3,#4){\ensuremath{\mathcal#1(#2,\,#3,\,#4)}} %
\def\ts#1#2{\tc{#1}(:,#2,:)} %
\def\ttens#1#2{\ensuremath{\mathcal #1(#2)}}

\long\def\todel#1{{\color{gray}#1}}
\long\def\todel#1{\relax}
\def\None{\texttt{None}}
\def\absD#1{\abs|\mathbb D_{#1}|}%

\title{Constructive TT-representation of the tensors given as index interaction functions with applications}
\author{Gleb Ryzhakov\\
Center for Artificial Intelligence Technology,\\
Skolkovo Institute of Science and Technology\\
Bolshoy Boulevard 30, bld. 1\\
Moscow, Russia 121205\\
  \texttt{g.ryzhakov@skoltech.ru} \\
  \\
Ivan Oseledets\\
Center for Artificial Intelligence Technology,\\
Skolkovo Institute of Science and Technology\\
Bolshoy Boulevard 30, bld. 1\\
Moscow, Russia 121205\\
and\\
AIRI, Moscow, Russia\\
  \texttt{i.oseledets@skoltech.ru} \\
}

\begin{document}
\maketitle
\begin{abstract}
This paper presents a method to build explicit tensor-train (TT) representations.
We show that a wide class of tensors can be explicitly represented with sparse TT-cores,
obtaining, in many cases, optimal TT-ranks.
Numerical experiments show that our method outperforms the existing ones in several practical applications, including game theory problems.
Theoretical estimations of the number of operations show
   that in some problems,
such as permanent calculation,
  our methods are close to the known optimal asymptotics,
    which are
    obtained by a completely different type of methods.
\end{abstract}

\section{Introduction}
The tensor train is a powerful tool for compressing multidimensional tensors
(by tensor we mean a multidimensional array of complex numbers).
It allows us to circumvent the curse of dimensionality in a number of cases.
In a case of $d$-dimensional tensor with number of indices equal to~$n$ for each dimension,
direct storage of tensor involves $O(n^d)$ memory cells,
while tensor train bypasses~$O(ndr^2)$, where $r$ is average rank of TT decomposition~\cite{Oseledets2011}.
In many important applications, the average rank may be small enough so that $n^d\gg ndr^2$.

Existing methods allow one to build TT-decompositions by treating the tensor values as a black box.
The TT-cross approximation method~\cite{oseledets2010ttcross}  adaptively queries the points where the tensor value is evaluated.
The iterative alternative schemes  such as 
alternating least squares method~\cite{Oseledets2012}
or
alternative linear schemes~\cite{Holtz2012}, build a decomposition 
consistently updating the decomposition cores.
These methods do not take into account the analytic dependence,
if any, of the tensor value on its indices.
At the same time, even for relatively simple tensors, these methods can build a TT decomposition
for a long time and in the vast majority of cases
obtain an answer with a given error greater than zero,
even if the original tensor has an exact TT decomposition.

In this paper, we present a fast method
to directly construct cores of the TT decomposition
of a tensor for which the analytical dependence of the tensor value on the values of its indices is known.
Technically, our method works with functions,
each of which depends on tensor index and which are sequentially applied to
the values of the previous functions.
This sequence of operations has a simple form in inverse Polish notation~\eqref{eq:Polish}.
This functions we call \emph{derivative functions} hereafter.
However,
this assignment covers quite a large range of functional dependences of tensor value on its indices
if such a set of functions is chosen skillfully.
Examples are given in the section~\ref{sec:Applications} and Appendix.

Our method works best in cases where the
derivative functions together with the tensor itself 
have a small number of possible values.
In the Application section and Appendix
there are several examples for \emph{indicator} tensors taking values only 0 and 1.

TT-cores, obtained by our method, are highly sparse,
which gives an additional gain in performance.
In many cases our method gives the lowest possible TT-rank,
so that no further rounding of the TT-cores is required.
In some other applications, the ranks of the TT decomposition
obtained by our method
can be substantially higher than those obtained by approximate methods.
However, in a large number of such cases, the sparse structure of the cores allows
one to achieve performance comparable to known algorithms.

The advantage of representing tensors in the TT format is not only in overcoming the curse of dimensionality,
but also in the implemented tensor algebra for them:
we can easily add, multiply, and round TT-tensors~\cite{Oseledets2011}.
In this way we can, for example, construct a set of indicator tensors that represent some constraints in the given problem in advance,
and then combine these constraints arbitrarily by multiplying these tensors with a data tensor.
As a practical use of such a scheme, 
we give an example of calculating the permanent of a matrix.
Other examples of various combinations of ,,simple'' tensors
constructed by our method
to solve applied problems
are given in the Appendix.
Python code with the examples is in the public domain\footnote{\url{https://github.com/G-Ryzhakov/Constructive-TT}}.

Our method has a direct extension to more complex cases of tensor networks,
for one of the cooperative games below and
in several examples is Appendix. %
Such a construction is called TT-Tucker~\cite{Dolgov2013,Oseledets2011}.

Our main contribution and advantages of our approach
\begin{itemize}
    \item
        the exact and fast representation of the tensor in TT-format, which can then, if necessary, be rounded to smaller ranks with a given accuracy.
In many of the given examples, this representation is optimal in the sense that the ranks of the TT decomposition cannot be reduced without loss of accuracy;
    \item 
        highly sparse structure of TT-decomposition cores which leads to a noticeable reduction in calculations;
    \item 
        an unified approach and a simple algorithmic interface to inherently different tasks and areas
        including those problems for which it is not immediately obvious the representation of the function specifying the tensor value in the form of consecutive functions~\eqref{eq:Polish};
\item
the ability to construct an approximate TT-decomposition with a controlled error or/and with the specified maximum ranks of the TT decomposition;
    \item 
        the possibility to build an iterative algorithm for calculating the value sought in a particular problem, without involving the notion of the tensor or matrix operations;
    \item 
        the ability to combine different conditions and quickly impose additional conditions when the basic tensor has already been constructed (and rounded).

\end{itemize}
\paragraph{Background}
Consider a tensor $\mathcal K$ with the following \emph{shapes} $\{n_1,\,n_2,\,\ldots,\,n_d\}$, \ie,
$
    \mathcal K\in\mathbb C^{n_1\times n_2\times\cdots\times n_d}
$,
where~$\mathbb C$ is the set of complex numbers.
TT-decomposition of the tensor $\mathcal K$ with 
\emph{TT-ranks}
$\{\rangeidxO rd\}$
is defined as the product
\begin{multline*}
    \mathcal K(i_1,\,i_2,\,\ldots,\,i_n)=
    \sum_{\alpha_0=1}^{1}
    \sum_{\alpha_1=1}^{r_1}
    \cdots
    \sum_{\alpha_{d-1}=1}^{r_{d-1}}
    \sum_{\alpha_d=1}^{1}
    \mathcal G_1(\alpha_0,\,i_1,\,\alpha_1)
    \mathcal G_2(\alpha_1,\,i_2,\,\alpha_2)
    \cdots\\\cdots
    \mathcal G_{d-1}(\alpha_{d-2},\,i_{d-1},\,\alpha_{d-1})
    \mathcal G_d(\alpha_{d-1},\,i_d,\,\alpha_d),
\end{multline*}
where tensors $\mathcal G_i\in\mathbb C^{r_{i-1}\times n_i\times r_{i}}$ are called \emph{cores} of the TT-decomposition (we let $r_0=r_d=1$).

\section{Building TT-representation with the given sequence of functions}\label{sec:main}
In this section we present our main result: we show how to directly build a TT-representation of the tensor
given in the inverse Polish notation as follows
\begin{equation}
    \mathcal K(i_1,\,i_2,\,\ldots,\,i_d)=
    (0 i_1 f^{(1)} i_2 f^{(2)} \cdots i_{l-1} f^{(l-1)}  )
    (0 i_d f^{(d)} i_{d-1} f^{(d-1)} \cdots i_{l+1} f^{(l+1)}  )
    i_lf^{(l)},
    \!
    \label{eq:Polish}
\end{equation}
where $\{f^{(i)}\}_{i=1}^d$ %
are given functions of two variables,
except for the \emph{middle} function~$f_l$, which has three arguments in the case $0< l < d$.
The brackets
have been added to the notation for convenience.

Our  task is to construct TT-decomposition of the tensor~$\mathcal K$ 
knowing only %
functions~$\{f^{(i)}\}_{i=1}^d$.

From an algebraic point of view, we want to build TT-decomposition of  such tensors~$\mathcal K$,
each element~$\mathcal K(i_1,\,i_2,\,\ldots,\,i_d)$ of which can be calculated in two consecutive passages.
One, from to the left of the right, is as follows
\begin{gather}
    a_1(i_1)=\f^1_{i_1}(0),
    \quad 
    a_2(i_1,\,i_2)=\f^{2}_{i_2}(a_1),
    \quad
    a_3(i_1,\,i_2,\,i_3)=\f^{3}_{i_3}(a_2),\\
    \cdots\\
    a_{l-1}(i_1,\,i_2,\,\cdots,\,i_{l-1})=\f^{l-1}_{i_{l-1}}(a_{l-2}),
    \label{eq:polish_a_left}
\end{gather}
then from right to left
\begin{gather}
    a_d(i_d)=\f^{d}_{i_d}(0),
    \quad
    a_{d-1}(i_d,\,i_{d-1})=\f^{d-1}_{i_{d-1}}(a_d),
    \\
    \cdots\\
    a_{l+1}(i_d,\,i_{d-1},\,\cdots,\,i_{l+1})=\f^{l+1}_{i_{l+1}}(a_{l+2}),
    \label{eq:polish_a_right}
\end{gather}
and, finally,
\begin{equation}
    \mathcal K(i_1,\,i_2,\,\ldots,\,i_d)=
    f^{(l)}(i_l,\,a_{l-1},\,a_{l+1}).
    \label{eq:polish_a_mid}
\end{equation}

The computation tree of this procedure is shown on Fig.~\ref{fig:comp_tree}.

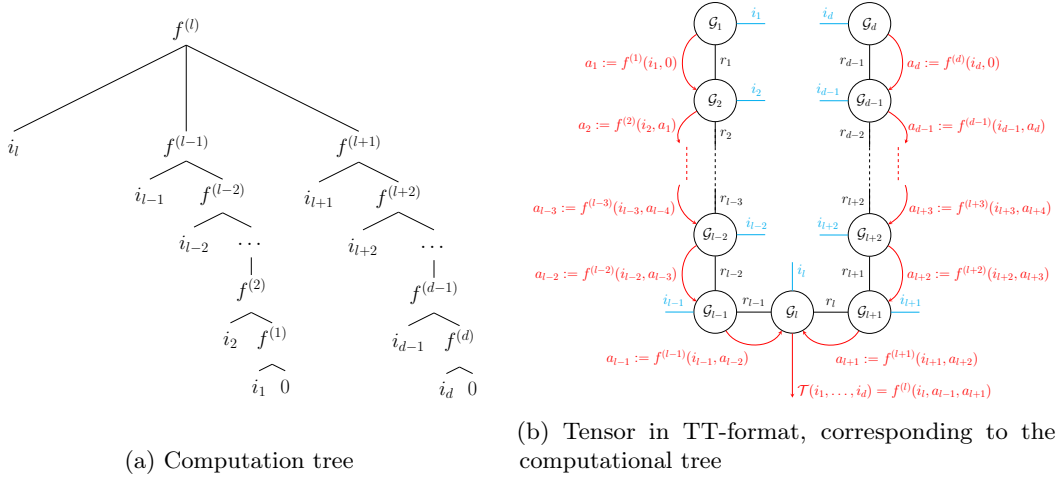
\begin{figure}[htpb]
    \centering
\begin{subfigure}[b]{0.45\textwidth}
\resizebox{\textwidth}{!}{
\centering
\LARGE
\Tree [ .\fnum l
$i_l$
[ .\fnum{l-1}
        $i_{l-1}$  
        [.\fnum{l-2}
        $i_{l-2}$
          [ .$\cdots$
           [ .\fnum2
        $i_2$
        [ .\fnum1
            $i_1$
            $0$
        ]
        ]
        ]
           ] 
] 
[ .\fnum{l+1}
        $i_{l+1}$
        [.\fnum{l+2}
           $i_{l+2}$ 
        [ .$\cdots$
        [ .\fnum{d-1}
        $i_{d-1}$
        [ .\fnum{d}
           $i_{d}$
           $0$
           ]
        ]
        ]
           ]
]
]
}
\vspace*{1em}
\caption{Computation tree}
\label{fig:comp_tree}
\end{subfigure}
\hfill
\begin{subfigure}[b]{0.51\textwidth}
\centering
\resizebox{\textwidth}{!}{%
\Large
\tikzstyle{arr} = [thick,red]
\tikzstyle{core} = [text width = 2em,  shape=circle]
\begin{tikzpicture}[scale=2.50,
every node/.append style = { align = center, minimum height = 10pt,
                                }, ]
\node[draw,core] (G1) at (0,-1) {$\mathcal G_{1}$};
\node[draw,core,below = 1.2 of G1] (G2) {$\mathcal G_{2}$};
\draw [thick] (G1) -- node[right] {$r_1$}  (G2);
\def\cutr{0.65}
\node[draw,core,below = 3.2 of G2] (Glm2) {$\mathcal G_{l-2}$};
\draw [thick] (Glm2) -- node[right] {$r_{l-3}$} ++(0,\cutr);
\draw [thick] (G2)  -- node[right] {$r_{2}$} ++(0,-\cutr);
\draw[dashed] (G2) ++(0,-0.4) -- (Glm2) ++(0,+0.4);
\node[draw,core,below = 1.2 of Glm2] (Glm1) {$\mathcal G_{l-1}$};
\draw [thick] (Glm2) -- node[right] {$r_{l-2}$} (Glm1);
\draw[arr] (G1) edge[out=-150,in=150,-stealth] node[left]   {$a_1:=f^{(1)}(i_1,0)$}  (G2);
\draw[arr] (G2) edge[out=-150,in=110,-stealth] node[left]    {$a_2:=f^{(2)}(i_2,a_1)$}  ++(-0.5,-0.6);
\draw[arr] (Glm2) edge[in=-110,out=150,stealth-] node[left] {$a_{l-3}:=f^{(l-3)}(i_{l-3},a_{l-4})$}  ++(-0.5,0.7);
\draw[arr] (Glm2) edge[out=-150,in=150,-stealth] node[left] {$a_{l-2}:=f^{(l-2)}(i_{l-2},a_{l-3})$}  (Glm1);
\draw[dashed,red,thick] (G2.west) ++(-0.1,-0.65) --  ++(0,-0.5);

\foreach \nd/\itext in {G1/i_1,G2/i_2,Glm2/i_{l-2}}
{
\draw [thick,cyan] (\nd) -- node[pos=0.7,above] {$\itext$} ++(0.7,0);
}
\node[draw,core,right = 1.2 of Glm1] (Gl) {$\mathcal G_{l}$};
\node[draw,core,right = 1.2 of Gl] (Glp1) {$\mathcal G_{l+1}$};
\draw [thick] (Glm1) -- node[midway,above] {$r_{l-1}$}  (Gl);
\draw [thick] (Gl) -- node[midway,above] {$r_{l}$}  (Glp1);
\draw [thick,cyan] (Glp1) -- node[pos=0.7,above] {$i_{l+1}$} ++(0.7,0);
\draw [thick,cyan] (Glm1) -- node[pos=0.7,above] {$i_{l-1}$} ++(-0.7,0);
\draw [thick,cyan] (Gl) -- node[pos=0.7,right] {$i_{l}$} ++(-0,0.7);
\draw[arr] (Glm1) edge[out=-60,in=-120,-stealth] node[below left]  {$a_{l-1}:=f^{(l-1)}(i_{l-1}, a_{l-2} )$}  (Gl);
\draw[arr] (Glp1) edge[out=-120,in=-60,-stealth] node[below right]  {$a_{l+1}:=f^{(l+1)}(i_{l+1}, a_{l+2} )$}  (Gl);
\draw[arr] (Gl) edge[pos=0.9,out=-90,in=90,-stealth] node[right]  {$\mathcal T(i_1,\ldots,i_d)=f^{(l)}(i_{l}, a_{l-1},a_{l+1} )$}  ++(0,-1.2);
\node[draw,core,above = 1.2 of Glp1] (Glp2) {$\mathcal G_{l+2}$};
\node[draw,core,above = 3.2 of Glp2] (Gdm1) {$\mathcal G_{d-1}$};
\node[draw,core,above = 1.2 of Gdm1] (Gd) {$\mathcal G_{d}$};
\foreach \nd/\itext in {Gd/i_d,Gdm1/i_{d-1},Glp2/i_{l+2}}
{
\draw [thick,cyan] (\nd) -- node[pos=0.7,above] {$\itext$} ++(-0.7,0);
}
\draw[dashed] (Gdm1) ++(0,-0.4) -- (Glp2) ++(0,+0.4);
\draw [thick] (Glp2) -- node[left] {$r_{l+2}$} ++(0,\cutr);
\draw [thick] (Gdm1)  -- node[left] {$r_{d-2}$} ++(0,-\cutr);
\draw [thick] (Gd) -- node[left] {$r_{d-1}$}  (Gdm1);
\draw [thick] (Glp2) -- node[left] {$r_{l+1}$}  (Glp1);

\draw[arr] (Gd) edge[out=-30,in=30,-stealth] node[right]   {$a_d:=f^{(d)}(i_d,0)$}  (Gdm1);
\draw[arr] (Gdm1) edge[out=-30,in=70,-stealth] node[right]    {$a_{d-1}:=f^{(d-1)}(i_{d-1},a_d)$}  ++(0.5,-0.6);
\draw[arr] (Glp2) edge[in=-70,out=30,stealth-] node[right] {$a_{l+3}:=f^{(l+3)}(i_{l+3},a_{l+4})$}  ++(0.5,0.7);
\draw[arr] (Glp2) edge[out=-30,in=30,-stealth] node[right] {$a_{l+2}:=f^{(l+2)}(i_{l+2},a_{l+3})$}  (Glp1);
\draw[dashed,red,thick] (Gdm1.east) ++(0.1,-0.65) --  ++(0,-0.5);
\end{tikzpicture}%
}
\caption{Tensor in TT-format, corresponding to the computational tree}
\label{fig:comp_tree_cores}
\end{subfigure}
\caption{Computation tree we can handle and the resulting TT-decomposition}
    \label{fig:comp_tree_ex}
\end{figure}

Each function~$f^{(k)}$ corresponds to a core in the TT-decomposition of the resulting tensor, see Fig.~\ref{fig:comp_tree_cores}.
In this figure, the cores of the expansion, which are 3-dimensional tensors, are represented by circles.
The blue lines with indices correspond to the input indices of the tensor.
The black lines connecting the cores
correspond to the dimensions,
by which the summation takes place.
The red arrows show the correspondence
between the output vector obtained by successive multiplication of the cores,
starting from the left or right end,
and the values of the derivative functions.

We write the first argument of the functions~$f$ as a lower index when describing examples and applications.

As an example,
consider step function in
the so-called 
\emph{Quantized Tensor Train decomposition} (QTT) (\cite{Oseledets2010}, \cite{khoromskij2011dlog})
when the tensor indices are binary $i_k\in\{0,\,1\}$
and all together represent a binary representation of some integer from the set~$\{0,\,1,\,\ldots,\,2^d-1\}$.
The value of the tensor represents the value of some given function~$P$ defined on this set,
\begin{equation}
    \mathcal I(i_1,\,i_2,\,\ldots,\,i_d)
    =
    P\Biggl(\sum_{j=0}^{d-1} i_{d-i}2^j \Biggr)
    \label{eq:QTT}.
\end{equation}
Function~$P$ is equal to the step function $P_{\text{step}}$ in this example:
\begin{equation}
        P_{\text{step}}(x)=
    \left\{
        \begin{aligned}
            0&,&x&\leq t,\\
            1&,&x&> t
        \end{aligned}
    \right.
\end{equation}
for the given integer number $t$, $0\leq t<2^d$.
Let the binary representation of~$t$ be
\begin{equation}
t=\sum_{j=0}^{d-1} b_{d-i}2^j.
\end{equation}
Then the form of the derivative functions for this tensor are depend only on the value of~$b_k$
and do not depend on the index~$k$ itself.
This function are the following: %
\begin{align}
    \text{If  } b_k&=0, &\text{ then }
    f^{k}_0(x)&:=f^{(k)}(0,\,x)=x,&\quad
    f^{k}_1(x)&:=f^{(k)}(1,\,x)=1;
\\
    \text{if  } b_k&=1, &\text{ then }
     f^{k}_0(x)&:=f^{(k)}(0,\,x)=
\left\{
    \begin{aligned}
        1&,&x&=1\\
        \text{None}&,&x&=0.
    \end{aligned}
\right.
,&\quad
    f^{k}_1(x)&:=f^{(k)}(1,\,x)=x.   
\end{align}

In our method,
the functions $f$ are predefined in the following way.
If,
in the process of calculating a tensor value,
the function~$f^{(k)}$ arguments are not in its domain,
we assume that it returns an empty value (we  denote this value by \None{} as in Python language).
The next function~($f^{(k-1)}$ or $f^{(k+1)}$), having received \None{}, returns also \None,
and so on, up to the ``middle'' function~$f^{(l)}$, which returns~0 if at least one of its arguments is \None.

In this example, the ``middle'' function is the last function: $l=d$, thus we %
consider it
as a function of two arguments.
The same is true for other examples in which the ``middle'' function is the first or last.

Note that in this example, the original analytic representation for the tensor did not assume pairwise interaction of the indices.
On the contrary, the formula~\eqref{eq:QTT} is quite integral:
its value depends on all variables at once.
Nevertheless, the expressions for the derivative function  turned out to be simple enough.
This situation holds for many examples as well, see Applications and Appendix.
Thus, our method can find a wide application.

It is worth noting that the arguments of the functions~$f$ have different physical meaning.
The first argument is the original natural number,
it corresponds to an index of the tensor under consideration.
By contrast, the second argument and the function value itself determine the relation between two indices of the tensor
and this relation can be complex.
In the general case,
middle-function~$f^{(l)}$ is complex-valued,
and the values of all other function~$f^{(k)}$, $k\neq l$,
can be of any nature for which a comparison operation is defined.

\subsection{TT-decomposition with the given derivative functions}

\begin{theorema}\label{th:main}
Let $\mathbb D_j$ be the image of the $j$-th function~$f^{(j)}$ 
from the derivative function set
for $1\leq j<l$ (``left'' function):
\def\defofD#1#2{\mathbb D_j=\bigl\{f^{(j#2)}(i,\,x)\colon 1\leq i\leq n_i,\,x\in\mathbb D_{j#11},\; \text{ if }f^{(j#2)} \text{ is defined at } (i,\,x)\bigr\} }%
\begin{equation}
    \defofD-{}, \quad j=\rangeo {l-1}
    .
\end{equation}
We let~$\mathbb D_0=\{0\}$.
Similarly for the ``right'' functions with $l<j\leq d$:
\begin{equation}
     \defofD+{+1}, \quad j=\range {l}{d-1}
    ,
\end{equation}
where we let~$\mathbb D_{d}=\{0\}$.
Then there exists TT-representation of the tensor~$\mathcal K$~\eqref{eq:Polish}
with TT-ranks~$r$ not greater than
\begin{equation}
    r=\{\absD0=1,\,\absD1,\,\ldots,\absD {d-1},\,\absD {d}=1\},
\end{equation}
where $\abs|\mathcal A|$ denotes the cardinality of a set~$\mathcal A$.
\end{theorema}
\begin{proof}
The proof is constructive.
The construction of the required cores of the TT-decomposition  takes place in two stages.

We first
enumerate elements of all images~$\{\mathbb D_i\}_{i=0}^d$
in arbitrary order, such that
we can address them by index.
Denote $\mathbb D_j[n]$
the $n$-th element of~$j$-th image,
$n=\range1\absD j$.
Now we can switch from the initial functions~$f$ with arbitrary values to functions~$\hat f$,
the range of each %
is 
a consecutive set of natural numbers starting from~$1$:
\def\fs#1{\text{index\_of}\left(f^{(j)}\bigl(i,\,\mathbb D_{j#11}[x]\bigr),\,\mathbb D_j\right)}%
\begin{align}
        \hat f^{(j)}(i,\,x)
    &:=\fs-,
    &
    j=\rangeo {l-1},&\;\;
    x=\rangeo \absD{j-1}\\
    \hat f^{(j)}(i,\,x)
    &:=\fs+,
    &
    j=\range{l+1}d,&\;\;
    x=\rangeo \absD{j+1}
    ,
    \label{eq:i_func_simple}
\end{align}
where function \texttt{index\_of} is defined as follows
\begin{equation*}
    z = \text{index\_of}(y,\,\mathcal A)
   \; \Longleftrightarrow\;
   y = \mathcal A[z]\;\;
   \text{ for some ordered set $\mathcal A$ and any $y\in\mathcal A$}.
\end{equation*}
We let $\text{index\_of}(\None,\,\mathcal A):=\None$.
Function~$\hat f^{(l)}$,
which corresponds to the ``middle'' function~$f^{(l)}$, is defined as follows
\begin{equation}
   \hat f^{(l)}(i,\,x,\,y):=
   f^{(l)}\bigl(i,\,
   \mathbb D_{l-1}[x],\,
   \mathbb D_{l}[y]
   \bigr)
\end{equation}

In the second stage,
we assign each integer input and output of new functions~$\{\hat f_i\}$
to a basis vector~$e$
\begin{equation}
    e(j,\,n)\in\mathbb R^n,
    \quad
    e(j,\,n)_i=\left\{
    \begin{aligned}
        1&,&\text{if } i=j&,\\
        0&,&\text{else}&
    \end{aligned}
    \right.
    ,
    \quad i,j=\rangeo n
    .
\end{equation}
Below we  omit the second argument in the notation of the basis vector~$e$
if it does not lead to ambiguity.
The basic idea is to construct the $j$-th ``left''  core ($j<l$) 
of the desired TT-decomposition 
corresponding to the function~$\hat f^{(j)}$ according to the following scheme:
\begin{equation}
\text{if }\;
y=\hat f^{(j)}(i,\,x),
\quad
\text{then }\;
e(x)^T\mathcal G_j(:,\,i,\,:)=e(y)^T,
\quad i=\rangeo n_j
\label{eq:func_matrix}
\end{equation}
where $\mathcal G_j(:,\,i,\,:)\in\mathbb R^{\absD {j-1}\times\absD j}$
denotes the matrix representing the $i$-th slice of $j$-th core.
The elements of this core are constructed explicitly:
\begin{equation}
\mathcal
G_j(x,\,i,\,y)=
\left\{
\begin{aligned}
    1&,&\text{if }\;y&=\hat f^{(j)}(i,\,x)\\
    0&,&&\text{else}
\end{aligned}
\right.,\;\;
x=\rangeo\absD {j-1},\;\;
y=\rangeo\absD j
    \label{eq:func_core_A}
\end{equation}
We do the same for the ``right'' cores or which~$i>l$,
except that multiplication on the basis vector takes place on the right
\begin{equation}
\text{if }\;
y=\hat f^{(j)}(i,\,x),
\quad
\text{then }\;
\mathcal G_j(:,\,i,\,:)e(x)=e(y),
\quad i=\rangeo n_j
.
\label{eq:func_matrix_right}
\end{equation}

Finally, we construct the middle-core~$\mathcal G_l$
which corresponds to the function~$f^{(l)}$:
\begin{equation}
    \mathcal G_l(:,\,i,\,:)=
\left\{
\begin{aligned}
    \hat f^{(l)}(i,\,x,\,y)&,&\text{if }\hat f^{(l)}&\text{ defined on }(i,\,x,\,y)\\
    0&,&&\text{else}
\end{aligned}
\right.,\;\;
x=\rangeo\absD {l-1},\;\;
y=\rangeo\absD l
.
\end{equation}

We summarize this two stages in Algorithms~1--2 in Appendix.%

Theorem statement follows from this construction:
after multiplying $m$ ``left'' constructed cores, $1\leq m<l$,
starting from the first one,
we get the following basis vector
\begin{equation}
    \mathcal G_1(:,\,i_1,\,:)\cdots \mathcal G_m(:,\,i_m,\,:)=e(a_m)^T
\end{equation}
where~$a_m$ is defined in~\eqref{eq:polish_a_left}. 
A similar basis vector~$e(a_p)$ is obtained by successive multiplication of all cores,
starting from the last one with the index~$d$ and up to some $p > l$
with~$a_p$ defined in~\eqref{eq:polish_a_right}.
Conclusively, 
the statement of the theorem follows from the relation 
$e(a_{l-1})\mathcal G_l(:,\,i_l,\,:)e(a_{l+1})=\mathcal K(i_1,\,i_2,\,\ldots,\,i_d)$
which is a consequence of the definition of the elements of the middle-core~$\mathcal G_l$ 
and which corresponds to the relation~\eqref{eq:polish_a_mid}.
\end{proof}

\subsection{Rank reduction}\label{sec:rank}
One way to reduce the TT-rank in the case when an image~$\mathbb D_i$ of a function~$f^{(i)}$ 
contains too many elements 
is to partition the image~$\mathbb D_i$ into several sets
and map the basis vector~$e$,
discussed in the second stage of the Theorem,
to one of these sets.
This is possible if the value of the function belongs to a space with a given topology.
In the simplest case, when the value of the function is real,
we can combine into one specified set only those elements from the image of the function for which $\abs|x-y|<\epsilon$, $x,y\in\mathbb D_i$ is satisfied
with the given~$\epsilon$.
In addition,
we can specify a maximum number of such sets,
thus fixing the maximum rank, increasing the  value of $\epsilon$ or combining some sets with each other.

Other ways of reducing the rank are described in the Appendix.

\subsection{Complexity}
The cores, except perhaps the middle-core,
obtained by our method
are highly sparse.
Each row of the slice~$\ts Gi$ of the core to the left of the middle-core consists of zeros,
except maybe one unit.
The same is true for the columns of the core to the right of the middle-core.
When multiplying a slice of a core by a vector,
we consider only those operations which do not involve addition with zero or multiplication by zero or unit.

When using the compressed format,
formally we do not need addition and multiplication to obtain a single tensor element,
since its calculation is reduced only to choosing the desired slice indices.

Consider a TT-tensor $\mathcal G$ obtained by our method with shapes $n=\left\{\rangeidx nd  \right\}$
and ranks $\left\{ \rangeidxO rd \right\}$
with the middle-core at the position~$l$, $1\le l\le d$.
To calculate the convolution of the tensor~$\mathcal G$ and an arbitrary rank-one tensor~$\mathcal W$ of the form
\begin{equation}
    \left\langle \mathcal G,\, \mathcal W \right\rangle =
    \sum_{\rangeidx r{d-1}}
    \;
    \sum_{\rangeidx nd}
    \mathcal G(1,\,n_1,\,r_1)
    \ldots
    \mathcal G(r_{d-1},\,n_d,\,1) %
    w(n_1)
    \cdots
    w(n_d),
    \label{eq:convolve_w}
\end{equation}
where $w(n)=\mathcal W(1,\,n,\,1)$,
we need
no more than~$n_\text{conv}$ additions
and
no more than~$n_\text{conv}$ multiplications
with
\begin{equation}
    n_\text{conv}= 
    \sum_{i=1}^{l-1}n_ir_{i-1}+
    \sum_{i=l+1}^{d}n_ir_{i}+
    n_m\bigl(r_{l-1}r_l+\min(r_{l-1},\,r_l)\bigr)
    \label{}.
\end{equation}
Indeed,
the first two sums in the last expression correspond to successive multiplication of the vector by the current core slice
for successive multiplication from each end of the tensor train up to the middle kernel with index~$l$.
The last term
corresponds to the multiplication of the two resulting vectors by the middle-core (which we assume dense) left and right.

\section{Applications}
\label{sec:Applications}
As a practical application of our method, in this section we give:
a) examples from the field of cooperative games, where we compare with existing methods and show our superiority both in execution time and in accuracy (our algorithm gives machine precision); b) the problem of calculating  matrix permanent,
where our method gives an estimated number of operations only twice as large as the optimized \emph{ad hoc} method of calculating the permanent using Hamiltonian walks.

\subsection{Cooperative games examples}\label{seq:coop_exampl}
As an example, consider several  so-called cooperative games~\cite{vonNeumann2007}.
Omitting the details of the economic formulation of the problem, 
let us briefly consider its mathematical model.
In general, in the theory of cooperative games it is necessary to calculate the following sum over all subsets of the given set~$\mathbb T$ of players
\begin{equation}
\pi(k)\defval
\sum_{\mathbb S\subseteq\mathbb T\setminus \{k\}}
p(\abs|\mathbb S|)
\bigl(
\nu(\mathbb S\cup \{k\})-
\nu(\mathbb S)
\bigr),
\quad
\text{ for all }\,
k\in\mathbb T.
\label{eq:games_semi}
\end{equation}
Here $p$ is some function of the number of players in a \emph{coalition}~$\mathbb S$. %
The function of a coalition $\nu$ is the function of interest, 
it depends on the game under consideration.
This function denotes the gain that a given coalition receives 
(\emph{value of the coalition}).

Below we briefly review several cooperative games
and compare the performance and accuracy of our algorithm
for solving them
with the algorithm presented in~\cite{Ballester2022},
which is based on the TT-cross approximation method
\footnote{We use the same setup as described in the cited paper for the experiments,
and take code from the open source repository \url{https://github.com/rballester/ttgames/} of the author of this paper.}.
For more details, we refer the reader to this paper and to the Appendix.
The results of the algorithm comparison are shown in Fig.~\ref{fig:games}.
\begin{figure}[thp]
\centering
\def\sb#1#2{
\begin{minipage}{0.23\linewidth}
\begin{subfigure}[b]{\linewidth}
\includegraphics[width=\linewidth]{images/times_for_coop_games_#1_times.pdf}
\caption{#2, times}
\label{fig:#1_times}
\end{subfigure}
\begin{subfigure}[b]{\linewidth}
\includegraphics[width=\linewidth]{images/times_for_coop_games_#1_res.pdf}
\caption{#2, accuracy}
\end{subfigure}
\end{minipage}
}
\sb{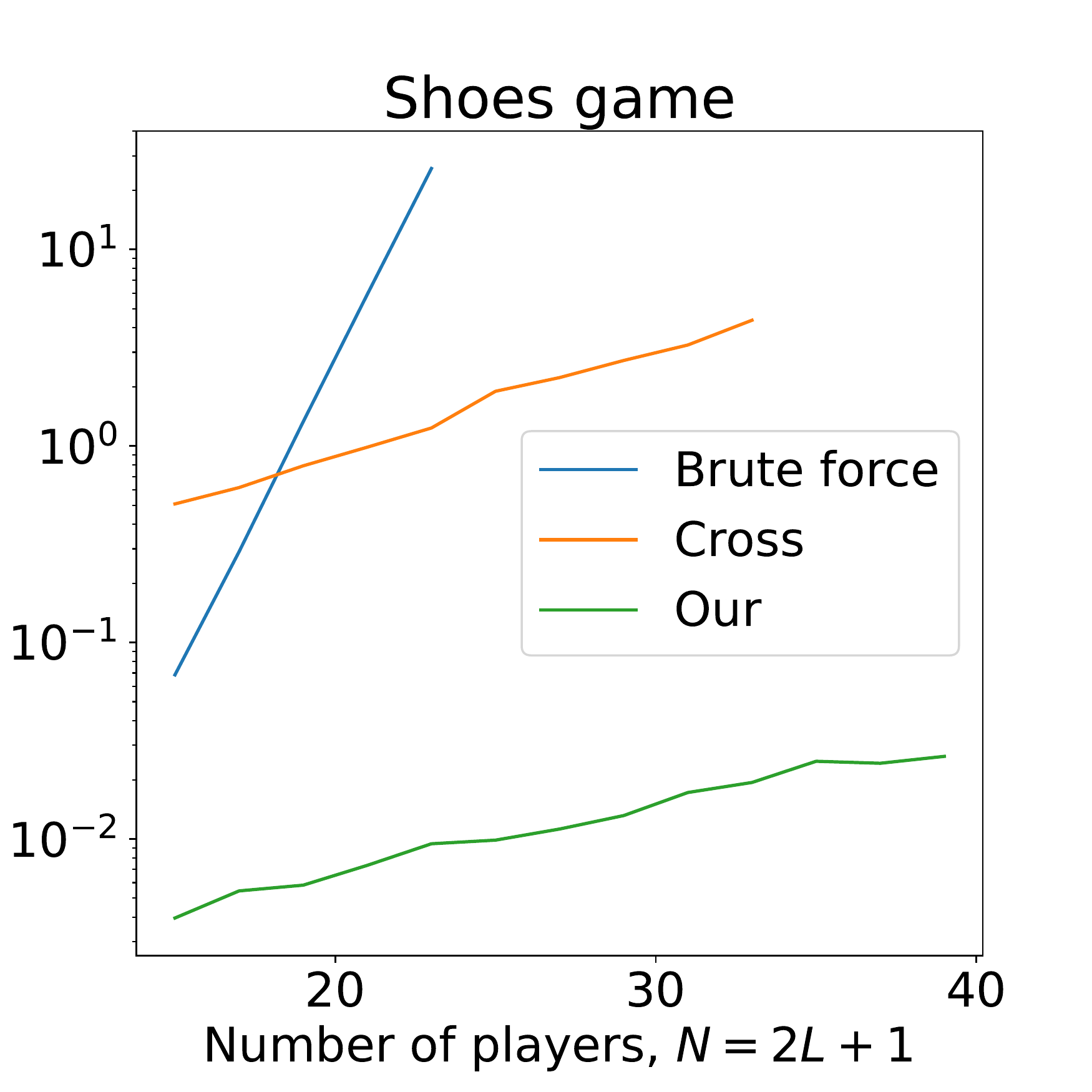}{Shoes game}
\sb{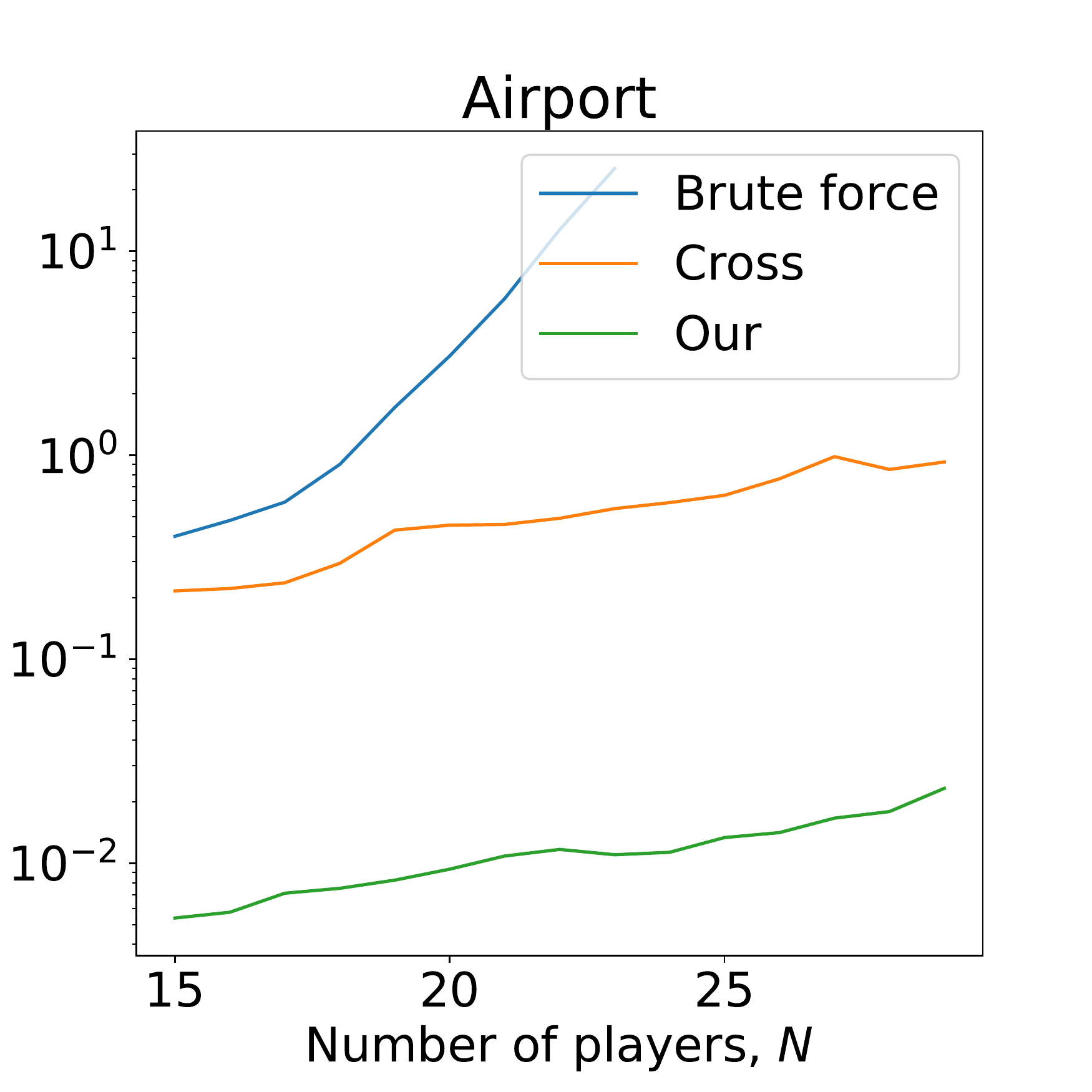}{Airport}
\sb{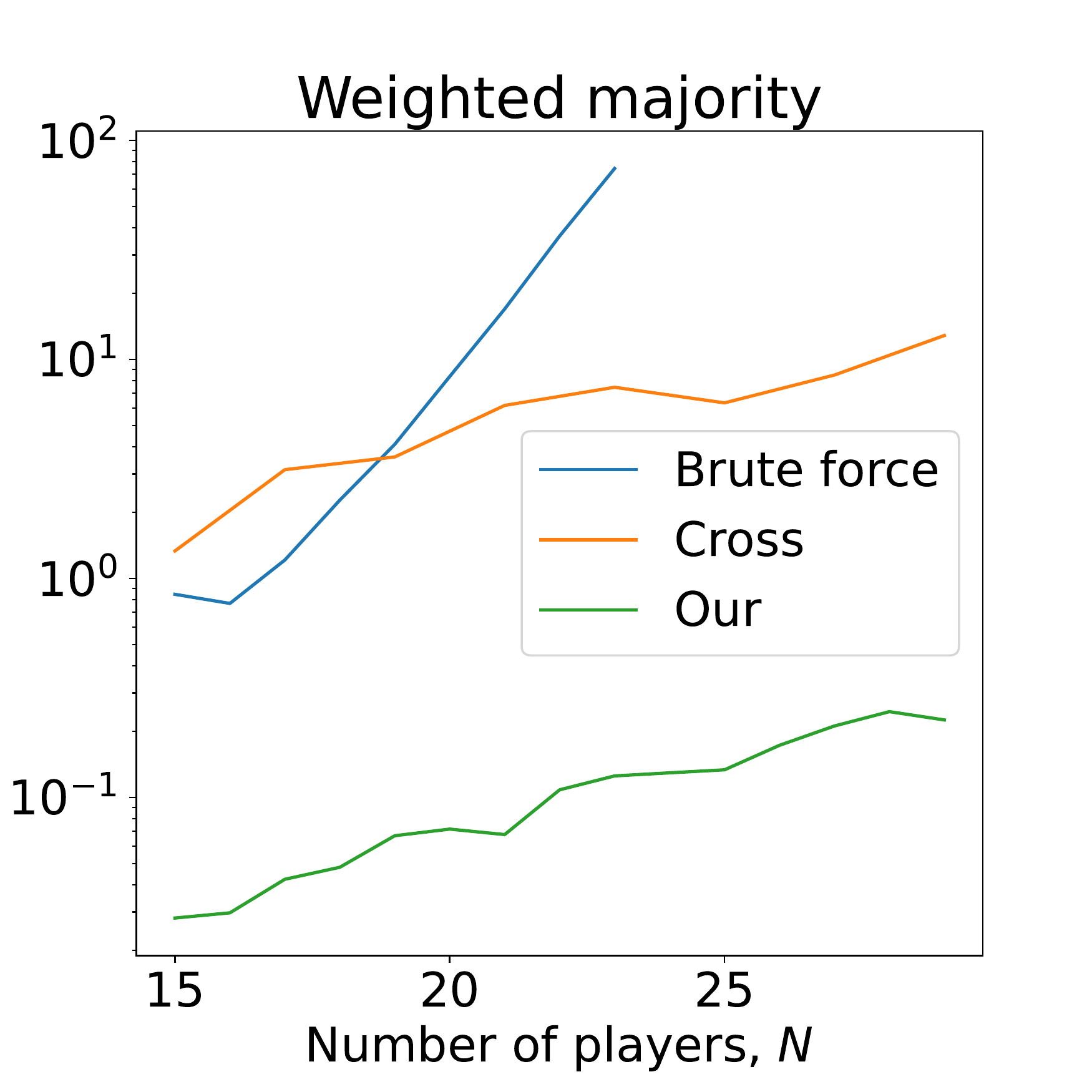}{Weighted majority game}
\sb{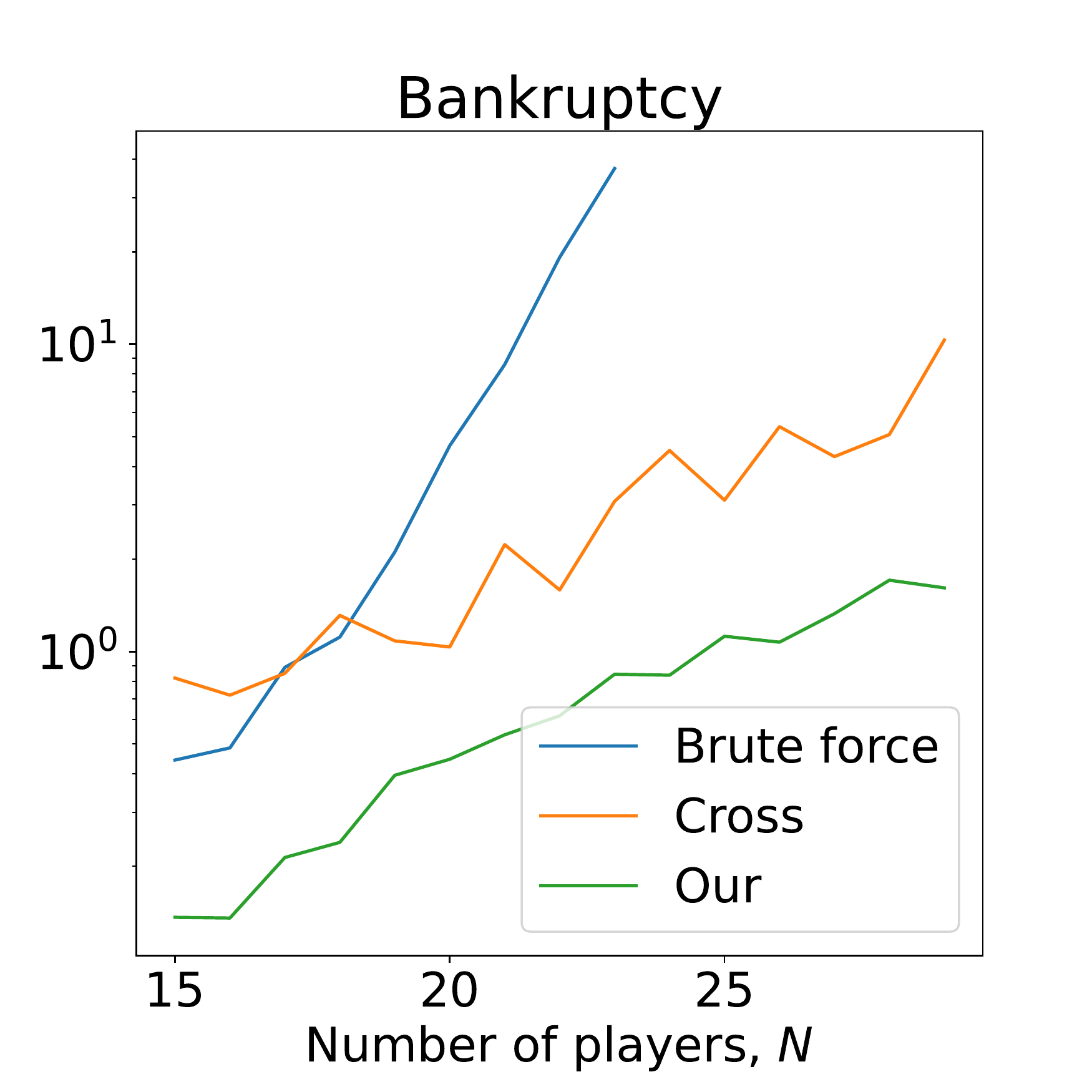}{Bankruptcy}
\caption{Times in seconds and relative accuracy 
as functions of number of players
for four cooperative games.
Brute force---calculating the sum~\eqref{eq:games_semi} directly,
Cross---results from the paper~\cite{Ballester2022}. %
}
\label{fig:games}
\end{figure}

\textbf{Shoe sale game.}
In this game,
participants are divided into two categories---those who sell left boots (indices~$1$ through~$L$)
and those who sell right boots (indices~$L+1$ through~$2L+1$).
As shoes can be sold only in pairs, the value of a coalition is the
minimum of the numbers of ''left'' and ''right'' players in a coalition.

Let us build tensors for this game.
To find the required value~$\pi$~\eqref{eq:games_semi}
in the case of cooperative games,
it is convenient to construct tensors that have a dimension equal to the number of players.
Each index of this tensor is binary: $1$~means a player is a member of a coalition, 
$0$~means he is not.

To construct the TT decomposition of the tensors
$p(\abs|\mathbb S|)
\nu(\mathbb S)
$
using our method,
let us take the following derivative functions:
\begin{equation}
    f_i^{(k)}(x)=x+i,\;\;
    1\leq k\leq d,\,
    k\neq L+1,
    \quad f_i^{(L+1)}(x,\,y+i)=\min(x,\,y)p(x+y+i),
    \;\;
    i=0,\,1,
    \label{}
\end{equation}
thus middle-core is placed on the position~$l=L+1$.
The derivative functions for constructing the tensor 
$p(\abs|\mathbb S|-1)
\nu(\mathbb S)
$
are selected in a similar way (we let $p(-1)=p(2L+1)=0$).
Once the cores for both tensors are constructed,
we can calculate the sum~\eqref{eq:games_semi} for different values of~$k$
by taking the corresponding slices of the cores by the method described in~\cite{Ballester2022}
and performing the convolution.
\textbf{Airport.}
It is not a cooperative game in its purest form,
as instead of gain we have a payoff,
but the mathematical essence is in the spirit of cooperative games.
Each player represents an aircraft that needs a landing strip of length~$c_k$.
Thus, $\nu(\mathbb S)=\max\{c_i \colon i\in\mathbb S\}$.
In order to construct a TT-representation of the tensor corresponding to this~$\nu$,
let us first order the values of~$c_k$ in descending order.
Then the derivative functions are
\begin{equation}
f^k_0(x)=x,\quad
f^k_1(x)=
\left\{
    \begin{aligned}
        x&,   &x&>0\\
        c_k&, &x&=0,
    \end{aligned}
\right.
,\quad 1\leq k\leq d.
\end{equation}
However, with these derivative functions the TT-ranks can get very large,
especially for non-integer~$c_k$.
To reduce the ranks, we do the following trick:
we  break the second function~$f^k_1$ into two terms, 
\begin{equation}
f^k_1(x)=
f^k_{(1)}(x)+
c_k\cdot f^k_{(2)}(x),
\quad
f^k_{(1)}(x)=
\left\{
    \begin{aligned}
        x&,   &x&>0\\
        \text{None}&, &&\text{else}
    \end{aligned}
\right.
,
\;\;
f^k_{(2)}(x)=
\left\{
    \begin{aligned}
        1&,   &x&=0\\
        \text{None}&, &&\text{else}
    \end{aligned}
\right..
\end{equation}
After that, 
we use three derivative functions to build TT-cores,
taking the multipliers~$c_k$ out of the build core as shown in Figure~\ref{fig:game_contr_max}.
This gives us a TT-Tucker format with TT-ranks equal to~$2$ and
matrices~$A_k$ with the following elements:
$
A_k=
\begin{pmatrix}
    1&0&0\\
    0&1&c_k
\end{pmatrix}
$.
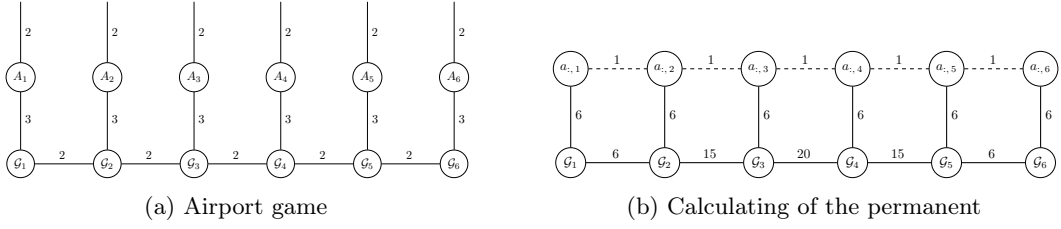
\begin{figure}[!tbh]
\centering
\begin{subfigure}[b]{0.44\textwidth}
\centering
\resizebox{\linewidth}{!}{%
\begin{tikzpicture}[scale=2.50]
    \def\maxd{6}
 \foreach \x in {1,...,\maxd}
      {
          \node[draw, shape=circle] (v\x) at (\x-1,0) {$A_{\x}$};
          \node[draw, shape=circle] (G\x) at (\x-1,-1) {$\mathcal G_{\x}$};
          \draw [thick] (v\x) -- node[right] {$3$} ++(0,-0.8) --  (G\x);
          \draw [thick] (v\x.north) -- node[right] {$2$} ++(0,+0.7);
      }
      \foreach \x in {1,...,\fpeval{\maxd - 1}}
      {
      \draw [thick] (G\x) -- node[above] {$2$} ++(0.8,0) -- (G\fpeval{\x + 1});
      }
\end{tikzpicture}%
}%
\caption{
    Airport game%
}
\label{fig:game_contr_max}
\end{subfigure}
\hfill
\begin{subfigure}[b]{0.48\textwidth}
\resizebox{\linewidth}{!}{%
\begin{tikzpicture}[scale=2.50]
 \foreach \x in {1,2,...,6}
      {
          \node[draw, shape=circle] (v\x) at (\x-1,0) {$a_{:,\,\x}$};
          \node[draw, shape=circle] (G\x) at (\x-1,-1) {$\mathcal G_{\x}$};
          \draw [thick] (v\x) -- node[right] {$6$} ++(0,-0.8) --  (G\x);
      }
\def\x{0}
\foreach \x/\r in {1/6,2/15,3/20,4/15,5/6}
      {
          \draw [dashed] (v\x) -- node[above] {$1$} ++(0.8,0) -- (v\fpeval{\x + 1});
          \draw [thick] (G\x) -- node[above] {$\r$} ++(0.8,0) -- (G\fpeval{\x + 1});
      }
\end{tikzpicture}%
}
\caption{
    Calculating of the permanent
}
\label{fig:permanent_net}
\end{subfigure}
\caption{Tensor network for two problems, $d=6$.
 The  numbers denote ranks (number of terms in sum).
The dashed line shows that at this point the rank is~$1$, which means that there is no summation and the corresponding cores are independent.
}
\end{figure}

\textbf{Other games.}
Fig.~\ref{eq:games_semi} also shows comparisons with the other two games.
In Weighted majority game tensor $\nu$ has the form
$\nu(\mathbb S)=1$ if $\sum_{i\in\mathbb S}w_i\geq M$ and~$0$ otherwise for given weights~$\{w_i\}$ and threshold~$M$.
In Bankruptcy game, $\nu(\mathbb S)=\max(0,\,E-\sum_{i\notin\mathbb S}c_i)$ for the given values of~$\{c_i\}$ and~$E$.
The derivative functions for these problems are chosen quite straightforwardly.

\subsection{Matrix calculus: Permanent}
Consider a task of calculating permanent of a matrix~$\{a_{ij}\}=A\in\mathbb C^{d\times d}$.
To solve this problem using the presented technique, let us construct two tensors in TT-format.
The first tensor~$\mathcal A$ will represent products of matrix~$A$ elements in the form
\begin{equation}
    \ttens A{\rangeidx id}=
    a_{i_1 1}
    a_{i_2 2}
    \cdots
    a_{i_d d}.
\end{equation}
This is rank-1 tensor and its cores~$\{\mathcal H_k\in\mathbb C^{1\times d\times1}\}_{k=1}^d$ are
$
    \tc H_k(1,i,1)
    =a_{ik}$,
     $i=\rangeo d
$.

The second tensor 
is an indicator tensor for such a set of indices,
in which all indices are different
\begin{equation}
    \ttens I{\rangeidx id}=
\left\{
\begin{aligned}
    1&,& &\text{ if all $\rangeidx id$ are different},\\
    0&,& &\text{ else}.
\end{aligned}
\right.
    \label{}
\end{equation}
The cores~$\mathcal G$ of this tensor are obtained using the following derivative functions
\begin{equation}
    f_i^k(x)=
    \left\{
    \begin{aligned}
        x+2^i&,& x\,\&\,2^i&=0,\\
        \None&,& & \text{ else}
    \end{aligned}
    \right.,
    \;\; k<d;
\quad
    f_i^d(x)=
    \left\{
    \begin{aligned}
        1&,& x\,\&\,2^i&=0,\\
        0&,& & \text{ else}
    \end{aligned}
    \right.
    ,
    \label{}
\end{equation}
where the ampersand sign stands for bitwise AND for integers.
In this scheme
the middle-core is the last ($d$-th) core.
The $k$-th bit of the input~$x$ of the derivative functions 
corresponds to the $i$-th value of the set $\{1,\,2,\,\ldots,\,d\}$:
if this bit is zero, then the $i$-th value of the index has not occurred before,
and thus this value is valid for the current index.
The function sets this bit in~$x$ and the value~$x$ is passed on.
If the bit has already been set,
the derivative function returns \None, since two identical indexes are forbidden.

The permanent value is equal to the convolution of tensor~$\mathcal I$
and tensor~$\mathcal A$.
Since the tensor~$\mathcal A$ is one-rank tensor,
we can look at the computation of the permanent as an contraction of the tensor~$\mathcal I$ with weights equal to the corresponding elements of the given matrix~$A$,
see Fig.~\ref{fig:permanent_net}.

After calculating the function~$f_i^k$ we get a number~$x$ in the binary representation of which
there are exactly~$k$ bits equal to one.
From this one can concludes that the corresponding rank~$r_k$ is equal to $r_k=\frac{d!}{(d-k)!k!}$.
Using relation~\eqref{eq:convolve_w} we can obtain an upper estimate on the total number of operations for the convolution of the given tensor
as $2\overline n_{\text{conv}}\sim(2^{N+1} N)$.
However, one can notice that on average half of the rows in the cores slices consist entirely of zeros, since index repetition is ``forbidden'':
$\mathcal G_k(i,\,j,\,:)=0$ at index~$i$ corresponds to the~$x$ with the $j$-th bit set.

Thus, $n_{\text{conv}}= 1/2\overline n_{\text{conv}} $ and
after the cores are built,
total number of operations~$n_{\text{tot}}$ (both additions and multiplications) required to obtain the result at these ranks has asymotics
\begin{equation}
n_{\text{tot}}=
    2n_{\text{conv}}
    \sim
    2^N N.
    \label{}
\end{equation}
This asymptotic is better
than the one that
can be obtained from the well-known Ryser's formula for calculating the permanent:
$
P(A)=
(-1)^N
\sum_{\mathbb S\subseteq\{1,2,\ldots,N\}}
(-1)^{\abs|\mathbb S|}
\prod_{i=1}^N
\sum_{j\in \mathbb S}a_{ij}
$.
When applied head-on, this formula requires~$O(2^{N-1}N^2)$ operations.
It is true that if one uses
Hamiltonian walk on $(N-1)$-cube (Gray code)
for a more optimal sequence of subset walks,
this formula will give asymptotic~$n_{\text{tot}}\sim(2^{N-1}N)$
which is only twice as good as ours~\cite[pp. 220--224]{CA1978}.

This is an example of a problem
where we can first pre-calculate the tensor~$\mathcal I$ with conditions and then reuse it with different data tensors~$\mathcal A$ containing elements of a particular matrix.

\subsection{Other examples}
Other examples are given in Appendix.
They include:
    simple examples for sum, where we explicitly show sparse TT-cores;
    other cooperative games,
    where we
    show how one can build iterative algorithm, based on our method;
    Knapsack problem (in several formulations), where we use existing algorithms to find the (quasi-) maximal element of the TT-tensor;
    Partition problem;
    Eight queens puzzle in several extended formulations;
    sawtooth sequence;
     standard Boolean satisfiability problem.

\section{Related works}
In addition to the works mentioned in the introduction, let us list the following.

In the paper \cite{Oseledets2012constr} explicit representations of several tensors with known analytical dependence of indices are presented. 
In the survey~\cite{Grasedyck2013} techniques for low-rank tensor approximation are presented, including TT-format.
In works~\cite{Cichocki2016part1}--\cite{Cichocki2017part2}
many examples of applying Tucker and Tensor Train decomposition to various problems including machine learning and data mining algorithms.
Tensor completion method described in the paper \cite{Steinlechner2016} uses Riemannian optimization for reconstruction a TT-format of a blackbox.

\section{Conclusions and future work}

We presented an algorithm for constructing the tensor in TT-format in the case when an explicit analytic dependence is given between the indices.
The cores of the obtained TT-tensor are sparse,
which speeds up manipulations with such a tensor.
Examples are given in which our method can be used to construct TT-representations of tensors encountered in many applied and theoretical problems.
In some problems, our representation yields an answer faster and more accurately than state-of-the-art algorithms.

As a limitation of our method, let us point out the fast growth of the ranks in the case when the derivative functions have a large size of their image set.
Although we have a rank restriction procedure for this case,
as plans for the future we specify an extension of our algorithm to accurately construct a TT-decomposition for such cases as well, if it is known to be low-ranked.

\FloatBarrier

\bibliographystyle{unsrt}
\bibliography{biblio.bib}

\newpage

\appendix
\section*{Appendices}

\addcontentsline{toc}{section}{Appendices}
\renewcommand{\thesubsection}{\Alph{subsection}}
\subsection{Algorithms}
Algorithms~\ref{a:build_TT_func_f}--\ref{a:build_TT_func_G} summarize the constructive construction of TT-decomposition cores, which is described in Theorem~\ref{th:main}.
Here, the function \texttt{order} orders the set in any way, $\dom$ denotes the domain of a function.
If a function returns \None{} on some set of its arguments, then we assume that it is not defined on that set.

\begin{algorithm}[tbh!]
    \caption{Construction of the integer-valued functions based on the given complex-valued functions}
    \label{a:build_TT_func_f}
    \begin{algorithmic}[1]
    \def\ff_#1^#2(#3){f^{#2}(#1,\,#3)}
        \Require{Middle-index $l$, set of functions $\{f^{(i)}\}_{i=1}^d$ of two variables (function~$f^{(l)}$ have 3 arguments) }
        \Ensure{Functions $\{\hat f_j^{(i)}\}$}

    \State{\AlgComm{Initialization}}
    \State{$n_k={}$maximum of the domain of the function $f^{(k)}$ on first variable for $k=\rangeo d$ (shapes of the resulting tensor)}
    \State{\AlgComm{Part I. Finding the function outputs~$R$ of each function. $R$ is a list of arrays.}}
    \State{$R[0]\gets \{0\}$, $R[d]\gets \{0\}$}
    \For{$i=1$ to $l-1$}
    \State{$R[i]\gets{}$order$(\{\ff_k^{(i)}(x)\colon k=\rangeo {n_i},\,x\in R[i-1]\})$}
    \EndFor
    \For{$i=d-1$ to $l$ step $-1$}
    \State{$R[i]\gets{}$order$(\{\ff_k^{(i)}(x)\colon k=\rangeo {n_i},\,x\in R[i+1]\})$}
    \EndFor
    \State{\AlgComm{Part II. Defining new functions}}
    \State{\AlgComm{From the left}}
    \For{$i=1$ to $l-1$}
    \For{$j=1$ to $n_i$}
        \For{$k$ in $R[i-1]$}
        \If{defined $f_j^{(i)}(k)$}
        \State{$x\gets \text{index\_of}\bigl(k,\,R[i-1]\bigr)$}
        \State{$y\gets \text{index\_of}\bigl(\ff_j^{(i)}(k),\,R[i]\bigr)$}
        \State{$\hat f^{(i)}_j(x)\defval y$}
        \EndIf
        \EndFor
    \EndFor
    \EndFor
    \State{\AlgComm{From the right}}
    \For{$i=d$ to $l+1$ step $-1$}
    \For{$j=1$ to $n_i$}
        \For{$k$ in $R[i]$}
        \If{defined $f_j^{(i)}(k)$}
        \State{$x\gets \text{index\_of}\bigl(k,\,R[i]\bigr)$}
        \State{$y\gets \text{index\_of}\bigl(\ff_j^{(i)}(k),\,R[i-1]\bigr)$}
        \State{$\hat f^{(i)}_j(x)\defval y$}
        \EndIf
        \EndFor
    \EndFor
    \EndFor

    \State{\AlgComm{Middle-index function}}
    \For{$j=1$ to $n_l$}
    \For{$k_1$ in $R[l-1]$, $k_2$ in $R[l]$}
        \If{defined $\ff_j^{(l)}(k_1,\,k_2)$}
        \State{$x_1\gets \text{index\_of}\bigl(k_1,\,R[l-1]\bigr)$}
        \State{$x_2\gets \text{index\_of}\bigl(k_2,\,R[l]\bigr)$}
            \State{$\hat f^{(l)}_j(x_1,\,x_2)\defval \ff_j^{(l)}(k_1,\,k_2)$}
        \EndIf
    \EndFor
    \EndFor
    \State{Return functions $\{\hat f_j^{(i)}\}$, (optionally) outputs $R$.}

\end{algorithmic}
\end{algorithm}

\FloatBarrier
\begin{algorithm}[tbh!]
    \caption{Explicit construction of the cores of the functional tensor}
    \label{a:build_TT_func_G}
    \begin{algorithmic}[1]
        \Require{Middle-index $l$, set of integer-valued functions $\{\hat f^{(i)}_j\}$ of one variable (function~$\{\hat f_j^{(l)}\}$ have 2 arguments)}
        \Ensure{Cores $\mathcal G_1,\,\ldots,\,\mathcal G_d$ of TT-decomposition of the functional tensor}

    \State{\AlgComm{Initialization}}
    \State{$n_k={}$maximum of the index~$j$ value of the function $\hat f_j^{(k)}$ for $k=\rangeo d$}
    \State{\AlgComm{From head}}
    \For{$i=1$ to $l-1$}
    \State{$\mathcal G_i \gets \text{zeros}\bigl(\max\limits_j\sup (\dom\hat f_j^{(i)}),\, n_i,\, \max\limits_{j,\,x}\hat f_j^{(i)}(x) \bigr) $ }
    \For{$j=1$ to $n_i$} \;\#\textit{Build each slice of the core as in~\eqref{eq:func_matrix}}
    \For{$x$ in $\dom\hat  f_j^{(i)}$}
    \State {$\mathcal G_i(x,\,j,\,\hat f_j^{(i)}(x))\gets1$}
    \EndFor
    \EndFor
    \EndFor
    \State{\AlgComm{From tail}}
    \For{$i=d$ to $l+1$ step $-1$}
    \State{$\mathcal G_i \gets \text{zeros}\bigl(\max\limits_{j,\,x}\hat f_j^{(i)}(x),\, n_i,\, \max\limits_j\sup (\dom\hat  f_j^{(i)}) \bigr) $ }
    \For{$j=1$ to $n_i$}
    \For{$x$ in $\dom f_j^{(i)}$}
    \State {$\mathcal G_i(\hat f_j^{(i)}(x),\,j,\,x)\gets1$}
    \EndFor
    \EndFor
    \EndFor
    \State{\AlgComm{Build middle-index core}}
    \State{$\mathcal G_l \gets \text{zeros}\bigl(
        \max\limits_j\max(\{x_1\colon(x_1,\,x_2) \in \dom\hat  f_j^{(l)} \})
        ,\, n_l,\,
        \max\limits_j\max(\{x_2\colon(x_1,\,x_2) \in \dom\hat  f_j^{(l)} \})
    \bigr) $ }
    \For{$j=1$ to $n_l$}
    \For{$(x_1,\,x_2)$ in $\dom\hat  f_j^{(l)}$}
            \State {$\mathcal G_l(x_1,\,j,\,x_2)\gets\hat f_j^{(l)}(x_1,\,x_2)$}
    \EndFor
    \EndFor

    \State{Return cores $\mathcal G_1,\,\ldots,\,\mathcal G_d$.}
\end{algorithmic}

\end{algorithm}

\subsection{Other applications}
\label{sec:o_Applications}

The examples in this section fall into two (possibly intersecting) broad types.
The first of them contains the calculation of the exact value of some quantity for which the TT decomposition of the tensor arising in the problem is needed.
In this case, we do not perform a TT tensor rounding
  (except when we reduce the rank using SVD decomposition with zero threshold,
  but we assume that after that the tensor value remains the same within machine accuracy).
From a technical point of view, usually the convolution operations (or other similar operations)
are performed without explicitly constructing the cores  of TT-decomposition completely in computer memory.
Instead, we use the functions obtained after applying the
Algorithm~\ref{a:build_TT_func_f}
directly as arrays of their values.
This is equivalent to working with matrices in a compressed format.

The second type of application consists of problems
in which rounding with a given accuracy is the advantage of the tensor approach---we get the answer with a certain error, but faster.

\subsubsection{Simple (model) examples}
\paragraph{Sum}\label{seq:sum}

Consider the following tensor~$\mathcal S$,
which is some function~$P$ of the sum of the elements of the given vectors~$a_1$, $a_2$, $\ldots$, $a_d$:
\begin{equation}
\ttens I{\rangeidx id}
=
P(a_1[i_1]+a_2[i_2]+\ldots+a_d[i_d]).
\end{equation}

We can easily build its TT-representation using the presented technique if we put the derivative functions equal to
\begin{equation}
f_i^k(x)=x+a_k[i].
\end{equation}
The view of these functions is the same for all cores except for the middle-core.
The middle-core
can stand in any place in this case, but it makes sense to put it in the middle of the tensor train (at the position~$\lfloor (d+1)/2\rfloor$)
to reduce the TT-ranks.
For the middle-core on the $m$-th place
the derivative functions are equal to
\begin{equation}
f_i^m(x,\,y)=P(x+y+a_m[i]).
\end{equation}

In the simple case, when the vector elements are consecutive numbers from zero to a given number: $a_i=\{0,\,1,\,\ldots,b_i\}$,
we have a natural correspondence between the value of the function~$y$ and the basis vector~$e^T(i)$ representing it:
\begin{equation}
y\Longleftrightarrow e^T(y+1).
\end{equation}
Thus, for example, 
the third slice (of any of  first $m-1$ cores) corresponding to the addition of number~$2$ will be of the form
\begin{equation}
\mathcal G(:,\, 3,\, :)=
\begin{pmatrix}
    0&0&0&0&0&\cdots &0&\cdots&0\\
    0&0&0&0&0&\cdots &0&\cdots&0\\
    1&0&0&0&0&\cdots &0&\cdots&0\\
    0&1&0&0&0&\cdots &0&\cdots&0\\
    0&0&1&0&0&\cdots &0&\cdots&0\\
     &&&\vdots \\
    0&0&0&0&0&\cdots &1&\cdots&0\\
\end{pmatrix},
\end{equation}
This matrix can be written in block form as
\begin{equation}
\mathcal G(:,\, 3,\, :)=
\begin{pmatrix}
    O_1 & O_2\\
    I & O_3\\
\end{pmatrix},
\end{equation}
where
$O_1$, $O_2$ and $O_3$
are zero matrices, and $I$ is the identity matrix.
If the result of the calculation of the previous function is~$3$,
which corresponds to the vector~$e_4^T$ as the result of multiplication of all previous kernels,
then after multiplying by the slice data we obtain $e(6)^T=e(4)^T\mathcal G(:,\, 3,\, :)$.
The basis vector~$e(6)^T$ expectedly corresponds to the value of the function equal to~$5$.

But in the more complex case,
the correspondence between the function value and its vector representation may not be obvious.
Consider the following first two vectors
$$
a_1=\{1,\,2,\,10\},
\quad
a_2=\{-1,\,0,\,5,\,8\}.
$$
Then the second slice of the second core $\mathcal G_2(;,\,2,\,:)$ which is generated by the function $f(x)=x+0$,
is equal to
$$
\mathcal G_2(;,\,2,\,:)=
\begin{pmatrix}
    0&1&0&0&0&0&0&0&0&0\\
    0&0&1&0&0&0&0&0&0&0\\
    0&0&0&0&0&0&0&1&0&0\\
\end{pmatrix}.
$$
This is true because  the first derivative function has three values, which correspond to the first three basis vectors:
$$
1\Longleftrightarrow e^T(1),\;\;
2\Longleftrightarrow e^T(2),\;\;
10\Longleftrightarrow e^T(3),
$$
whereas the value area of the second derivative function has 8 elements, which correspond to
$$
0 \Longleftrightarrow e^T(1),\;
1 \Longleftrightarrow e^T(2),\;
2 \Longleftrightarrow e^T(3),\;\;\ldots,\;\;
10\Longleftrightarrow e^T(7),\;
15\Longleftrightarrow e^T(8),\;
18\Longleftrightarrow e^T(9).
$$

Note that in the examples above we chose to assign basis vectors to the values of the derived functions according to their their ascending sorting.
This is not a crucial point, since this correspondence is conditional, it can change and be different for each core.

In the degenerate case, when the function~$P$ is identical: $P(x)=x$
we can construct the desired TT-decomposition with ranks equal to 2.
Namely, in this case the cores have the following explicit form
\begin{equation}
    \ts{G_1}i
=\bigl(1,\,a_1[i] \bigr);
    \quad
    \ts{G_k}i=
\begin{pmatrix}
    1 & a_k[i]\\
0 & 1
\end{pmatrix},
\;\;2\leq k\leq d-1;
    \quad
    \ts{G_d}i=
\begin{pmatrix}
    a_d[i]\\
1
\end{pmatrix}.
\label{eq:ex:sum}
\end{equation}
These cores can be constructed using our techniques as follows.
Consider the following tensor
with binary indices $i_k\in\{0,\,1\}$ 
which is equal to~$1$ iff only one of its indices is~$1$:
\begin{equation}
    \ttens I{i_1,\,i_2,\,\ldots,\,i_d}= 
\left\{
\begin{aligned}
    1&,& \sum_{k=1}^{d}i_k&=1,\\
    0&,& &\text{else}
\end{aligned}
\right..
\end{equation}
We can construct its cores~$\widehat{\mathcal G}$  using the following derivative functions
$$
f_0^k(x)=x,\quad
f_1^k(x)=
\left\{
\begin{aligned}
    1&,& x&=0,\\
    \text{None}&,& &\text{else}.
\end{aligned}
\right.
$$
And consider one-rank tensor~$\mathcal H$ with values to be
summed up with the following cores
$$
\ts{K_k}0=(1)
,\quad
\ts{K_k}1=(v_k).
$$
For this tensor it is true that
$$
\ttens H{{\underbrace{0,\,0,\,\ldots,\,0}_{k-1},\,1,\,0,\,\ldots,\,0}}= v_k,
$$
thus
its convolution with the tensor~$\mathcal I$
gives the sum of elements of~$v$:
\begin{equation}
    \sum_{i_1=0}^{1}
    \sum_{i_2=0}^{1}
    \cdots
    \sum_{i_d=0}^{1}
    \ttens I{i_1,\,i_2,\,\ldots,\,i_d}
    \ttens H{i_1,\,i_2,\,\ldots,\,i_d}
    =
    \sum_{k=1}^{d}v_k.
    \label{eq:sum_convolv}
\end{equation}
\begin{figure}[!htb]
\def\maxd{5}
\centering
\begin{subfigure}[b]{0.45\textwidth}
         \centering
\resizebox{\linewidth}{!}{%
\begin{tikzpicture}[scale=2.50]
     \foreach \x in {1,2,...,\maxd}
      {
          \node[draw, shape=circle] (v\x) at (\x-1,0) {$c_{\x}$};
          \node[draw, shape=circle] (G\x) at (\x-1,-1) {$\widehat{\mathcal G_{\x}}$};
          \draw [thick] (v\x) -- node[right] {$2$} ++(0,-0.8) --  (G\x);
      }
      \foreach \x in {1,...,\fpeval{\maxd - 1}}
      {
      \draw [thick] (G\x) -- node[above] {$2$} ++(0.8,0) -- (G\fpeval{\x + 1});
      }
\end{tikzpicture}%
}%
\caption{Representation of the sum~\eqref{eq:sum_convolv}, $c_k=\{1,\,v_k\}$}
\label{fig:sum_simple:a}
\end{subfigure}
\hfill
\begin{subfigure}[b]{0.45\textwidth}
         \centering
\resizebox{\linewidth}{!}{%
\begin{tikzpicture}[scale=2.50]
     \foreach \x in {1,2,...,\maxd}
      {
          \node[draw, shape=circle] (v\x) at (\x-1,0) {$C_{\x}$};
          \node[draw, shape=circle] (G\x) at (\x-1,-1) {$\widehat{\mathcal G_{\x}}$};
          \draw [thick] (v\x) -- node[right] {$2$} ++(0,-0.8) --  (G\x);
          \draw [thick] (v\x.north) -- node[right] {$n_\x$} ++(0,+0.7);
      }
      \foreach \x in {1,...,\fpeval{\maxd - 1}}
      {
      \draw [thick] (G\x) -- node[above] {$2$} ++(0.8,0) -- (G\fpeval{\x + 1});
      }
\end{tikzpicture}%
}%
\caption{Obtaining $\mathcal G$ via the operation~\eqref{ex:conv:G-C}}
\label{fig:sum_simple:b}
\end{subfigure}
\caption{Tensor network for  TT-tensor for simple sum for $d=\maxd$.
}
\label{fig:sum_simple}
\end{figure}
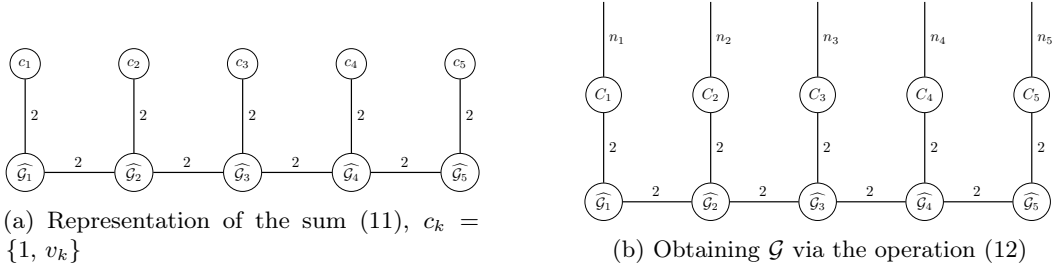
This convolution operation is shown schematically in Fig.~\ref{fig:sum_simple:a}.

Now, to be able to select the terms by choosing the indices of the tensor~$\mathcal G$,
we replace the vectors~$a$ in this network with matrices~$A$ of the following form
\begin{equation}
    C_k=
    \begin{pmatrix}
        1&1&1&\cdots&1&1\\
        v_k[1]&v_k[2]&v_k[3]&\cdots&v_k[n_k-1]&v_k[n_k]\\
    \end{pmatrix},
\end{equation}
where~$n_k$ is the length of the vector~$a_k$.
Finally, we obtain the cores~$\mathcal G_k$  by the convolution of the cores ~$\widehat{\mathcal G_k}$ and the matrices~$C_k$ (see Fig.~\ref{fig:sum_simple:b}.):
\begin{equation}
    \ts{G_k}i= \sum_{j=0}^1 \widehat{\mathcal G}(:,\,j,\,:) C_k(j,\,i).
\label{ex:conv:G-C}
\end{equation}

\paragraph{Step function in QTT format}%
\todel{%
In this example, consider the so-called QTT representation
when the tensor indices are binary $i_k\in\{0,\,1\}$
and all together represent a binary representation of some integer from the set~$\{0,\,1,\,\ldots,\,2^d-1\}$.
The value of the tensor represents the value of some given function~$P$ defined on this set:
\begin{equation}
    \mathcal I(i_1,\,i_2,\,\ldots,\,i_d)
    =
    P\left(\sum_{j=0}^{d-1} i_{d-i}2^j \right).
\end{equation}
Let $P$ be a step function:
$$
P(x)=
\left\{
    \begin{aligned}
        0&,&x&\leq t,\\
        1&,&x&> t
    \end{aligned}
\right.
$$
for the given integer number $t$, $0\leq t<2^d$.
Let the binary representation of~$t$ be
\begin{equation}
t=\sum_{j=0}^{d-1} b_{d-i}2^j.
\end{equation}
Then the form of the derivative functions will depend only on the value of~$b_k$.
If $b_k=0$, then
\begin{equation}
    f_k^0(x)=x
    ,\quad
    f_k^1(x)=1,
\end{equation}
and if 
$b_k=1$, then
\begin{equation}
    f_k^0(x)=
\left\{
    \begin{aligned}
        1&,&x&=1\\
        \text{None}&,&x&=0.
    \end{aligned}
\right.
    ,\quad
    f_k^1(x)=x.
\end{equation}
The same derivative functions will be for the middle-core
which is the last core
in this case.
}%
The definition of step function in QTT format and its 
derivative functions  are given in Section~\ref{sec:main}.
Let us consider in more detail how these functions give the desired value of the tensor built on them.

If for $j<k$ for some~$k$ values of~$b_j$ and~$i_j$ coincide: $b_j=i_j$, 
and for $k$-th  bit $b_k=0$ but $i_k=1$,
it means that $\mathcal I=1$ 
regardless of the lower bits. %
The value of the function~$f_k$ is~$1$.
This value will be ``carried'' to the end without change as for all functions~$f(1)=1$ 
giving the resulting value 1 of the TT-tensor.
If, on the contrary, $b_k=1$ and $i_k=0$
then we surely know that~$P=0$. 
The corresponding function is undefined at this value of the argument, which we symbolically write as $f_k^{(0)}(0)=\None$.
In the language of vectors, this means that
after multiplication by the $k$-th core
we get a zero vector,
not a basis vector,
so that for any subsequent values of the indices the result will be zero.
In the remaining case when $b_k=i_k$ for all $1\leq k\leq d$,
functions leave the argument unchanged.

The maximum TT-rank value in this example is no more than~$2$.
For the case of rank equal to 2, the explicit form of the cores are the following,
for $b_k=0$:
$$
G_k(:,\,0,\,:)=
\begin{pmatrix}
    1&0\\
    0&1\\
\end{pmatrix},
\;\;
G_k(:,\,1,\,:)=
\begin{pmatrix}
    0&1\\
    0&1\\
\end{pmatrix},
$$
and
for $b_k=1$:
$$
G_k(:,\,0,\,:)=
\begin{pmatrix}
    0&0\\
    0&1\\
\end{pmatrix},
\;\;
G_k(:,\,1,\,:)=
\begin{pmatrix}
    1&0\\
    0&1\\
\end{pmatrix}.
$$
\subsubsection{Cooperative games}
In this subsection, 
we take a closer look at the examples that were briefly given in Section~\ref{seq:coop_exampl}, 
and we also give some more examples of applying our methods to cooperative games, which are described in~\cite{Ballester2022}.

\todel{
Theory of cooperative games is
one of the areas where multidimensional tensors can be highly compressed in TT-format~\cite{Ballester2022}.
In this section we show only the possibility of fast compression of tensors by our methods;
for relevance and further detailed explanations of cooperative games we refer the reader to the cited article.

In general, in the theory of cooperative games it is necessary to calculate the following sum over all subsets of the set~$\mathbb T$ of players
\begin{equation}
\pi(k)\defval
\sum_{\mathbb S\subseteq\mathbb T\setminus \{k\}}
p(\abs|\mathbb S|)
\bigl(
\nu(\mathbb S\cup \{k\})-
\nu(\mathbb S)
\bigr)
\label{eq:games_semi_2}
\end{equation}
Here $p$ is some function of the number of players in coalitions~$\mathbb S$ and the corresponding tensor can be easily build (see ?????).
\todo{add ref to the TT}
The function of the coalition $\nu$ is the function of interest, it depends on the game under consideration.
This function denotes the gain that a given coalition receives (value of the coalition).

Below we
consider several cooperative games, the definition of the~$\nu$ function for them, and ways of constructing the TT-tensor that corresponds to it.
This TT-tensor has $\abs|\mathbb T|$
indices,
each index takes values~$0$ or~$1$,
indicating whether the player with the given number is part of the coalition.

}
\todel{
\paragraph{Shoe sale game}
In this game,
participants are divided into two categories---those who sell left boots (indices~$1$ through~$L$)
and those who sell right boots (indices~$L+1$ through~$2L+1$).
As shoes can be sold only in pairs, the value of a collation is the
minimum of the numbers of ''left'' and ''right'' players.

The tensor that corresponds to the function in question in this case is built head-on.
Both left and right functions has the form
$$
f_0(x)=x,\quad
f_1(x)=x+1,
$$
and the middle functions set placed on the $(L+1)$ position is the following
$$
g_0(x,\,y)=\min(x,\,y),\quad
g_1(x,\,y)=\min(x,\,y+1).
$$
The middle set of function corresponds to the ''right'' player, thus we add~$1$ to the argument~$y$ in the definition of~$g_1$.

The TT-cores built on such derivative functions have a maximum rank~$L+1$ without truncation and an effective rank $r_e\approx 17/(4\sqrt6)+\sqrt{2/3}L$.
\todo{define erank}
It is the smallest possible maximum rank
for a given sequence of indices, \ie{} it can not be reduce by truncation without loss of accuracy.

We can reduce the ranks so that they become the smallest possible in this sequence of players.
Namely, let the L-th place be occupied by the cores generated by the following functions
$$
f^L_i(x)=
\left\{
    \begin{aligned}
        x+i&,&\quad x+i&>0\\
        \text{None} &,& x+i&=0,
    \end{aligned}
\right.
,\quad i=0,\,1.
$$
The set of possible values at the output of these functions is equal to $\{1,\,2,\,\ldots,\,L\}$, excluding zero, so the rank is reduced by 1.
The resulting TT-ranks are $r=\{1,\,2,\,\ldots,\, L-1,\,L,\,L,\,L+1,\,L,\,L-1,\,\ldots,1\}$.
}

\paragraph{Shoe sale game}
Brief game conditions and
the derivative functions for this game are described in general form above.
Consider what happens when a certain~$k$ is given, \ie{} the order number of the player for whom we calculate the payoff~$\pi(k)$.
For each~$k$ we
rebuild  tensors
$p(\abs|\mathbb S|)
\nu(\mathbb S\cup \{k\})$ and
$p(\abs|\mathbb S|)
\nu(\mathbb S)
$.
At the $k$-th place we leave only one function of the two: $f^{(k)}_1$ for the first specified tensor, and $f^{(k)}_0$ for the second.
Thus, although we formally have a tensor of dimension~$d$, 
there is actually no summation at the $k$-th index, 
and we obtain a sum over all subsets of the set~$\mathbb T\setminus \{k\}$.
In addition, 
in the case of the  tensor
$p(\abs|\mathbb S|)
\nu(\mathbb S\cup \{k\})$, we  subtract a unit from the argument~$p$
since in this case 
$p(\abs|\mathbb S|)
\nu(\mathbb S\cup \{k\})
=p(\abs|\mathbb S\cup \{k\}|-1)
\nu(\mathbb S\cup \{k\})
$. 

Then we contract each tensor with ones,
namely, we calculate the following expression
\begin{equation}
    \def\oneterm#1#2#3{\left(\sum_{i_{#1}=1}^{n_{#1}}\mathcal G_{#1}(#2,\,i_{#1},\,#3)\right)}
    \oneterm11:
    \oneterm2::
    \cdots
    \oneterm d:1
    \label{eq:convl}
\end{equation}
from left to right in sequence.
In the case of this example, all $n_j=2$, $j=\rangeo d$.
Thus we build and contract~$2N=4L+2$ tensors, %
but since we do not build full cores in memory
and instead work in a compressed format, we get the computation times shown in Fig.~\ref{fig:shoes_times}.

\paragraph{Airport}
Brief game conditions and
the derivative functions for this game are described in general form above.

In this problem we once explicitly build the cores for the tensor~$\nu$, 
making a convolution with matrices~$A_k$ (see Fig.~\ref{fig:game_contr_max}).
This convolution
in fact, is reduced to the replacement of the corresponding unit in the core slice by the value of~$c_k$.
Separately, we construct the cores for the tensor~$p$, which is the given function of the sum of all indices equal to~$1$.
The construction of such a tensor is described at the beginning of Sec.~\ref{seq:sum}; 
in this case we  put $a_j[0]=0$,  $a_j[1]=1$, $j=\rangeo d$.

To calculate the value of~$\pi(k)$,
we applied a trick similar to the one used in~\cite{Ballester2022}.
Namely,
for each~$k$
we have left the core~$\mathcal G_k^{(p)}$
from the TT-decomposition to the tensor~$p$
only the first slice of dimensionality, which corresponds to the index~$i_k=0$,
so the following is true for the new core~$\widehat {\mathcal G_k}^{(p)}$:
\begin{equation}
    \widehat {\mathcal G_k}^{(p)}(:,\, 0,\,:)=
    \mathcal G_k^{(p)}(:,\, 0,\,:), \quad
    \widehat {\mathcal G_k}^{(p)}\in\mathbb R^{r^{(p)}_{k-1}\times1\times r^{(p)}_k}.
\end{equation}

For the core~$\mathcal G_k^{(\nu)}$ from the TT-decomposition of the tensor~$\nu$
we also removed the second slice and took the slice difference of the original core as the first slice:
\begin{equation}
    \widehat {\mathcal G_k}^{(\nu)}(:,\, 0,\,:)=
    \mathcal G_k^{(\nu)}(:,\, 1,\,:)-
    \mathcal G_k^{(\nu)}(:,\, 0,\,:),
    \quad
    \widehat {\mathcal G_k}^{(\nu)}\in\mathbb R^{r^{(\nu)}_{k-1}\times1\times r^{(\nu)}_k}.
\end{equation}

For the final result,
we have taken a convolution of these tensors,
which in TT-format is written as
\def\oneterm#1#2#3{\left(\sum_{i_{#1}=1}^2
    {\mathcal G^{(p)}_{#1}}(#2,\,i_{#1},\,#3)
    \otimes 
    {\mathcal G^{(\nu)}_{#1}}(#2,\,i_{#1},\,#3)
    \right)}
\begin{multline}
    \pi(k)=
    \oneterm11:
    \oneterm2::
    \cdots
    \\
    \cdots
    {\mathcal G^{(p)}_{k}}(:,\,1,\,:)
    \otimes
    {\mathcal G^{(\nu)}_{k}}(:,\,1,\,:)
    \cdots
    \oneterm d:1,
\end{multline}
where $\otimes$ denotes the Kronecker product of matrices.

For the numerical experiments we take values $c_k$ as
i.i.d.\ random values uniformly distributed on the interval~$[0,\,1]$.
\paragraph{Weighted majority game}
For this game, briefly described in Section~\ref{seq:coop_exampl},
we took the following derivative functions
for the tensor
$p(\abs|\mathbb S|)(
\nu(\mathbb S\cup \{k\}) - 
\nu(\mathbb S))
$ of dimension $d=\abs|\mathbb T|-1$:
\def\funj#1{
    f_0^j(x)=x,\quad
    f_1^j(x)=
    \left\{
    \begin{aligned}
        \{x[1]+w_{j#1},\,x[2]+1\}&,& x[1]+w_{j#1}&\leq M,\\
        \None&,& \text{else}
    \end{aligned}
    \right.,
}
\begin{equation}
    \funj{}
    \quad
    1\leq j <k,
\end{equation}
\begin{equation}
    \funj{+1}
    \quad
    k\leq j <d,
\end{equation}
and for the middle-function, which is the last one in this example, we have
\begin{equation}
    f_i^d(x)=
    \left\{
        \begin{aligned}
            p(x[2]+i)&,& x[1]+iw_{d+1}+w_k&\geq M>x[1]+iw_{d+1},\\
            \None&,& \text{else}
        \end{aligned}
    \right.,
    \quad
    i=0,\,1.
\end{equation}
(In the case of~$k=\abs|\mathbb T|=d+1$, we take~$w_{d}$ instead of $w_{d+1}$).
Note that in this example
the derivative functions are defined on a set of vectors of length~$2$. 
The first component of this vector accumulates a sum of weights~$\{w_j\}$ to compare with the threshold~$M$.
The second component of the input vector counts the number of players in the coalition~$\mathcal S$,
this value is passed to the function~$p$
in the last derivative function.

If at step~$j$, $j<d$ it turns out that
the accumulated sum of weights already exceeds the threshold~$M$ ($x[1]+w_j>M$)
then the derivative function returns \None,
thus zeroing out the value of the tensor.
Indeed, in this case, the difference $
\nu(\mathbb S\cup \{k\}) - 
\nu(\mathbb S))
$ will be obviously equal to zero regardless of what other coalition members are added, i.e. regardless of  $i_l$, $l>j$ index values.
This trick reduces the TT-ranks of the resulting tensors and thus reduces the execution time.

In this example, we construct the cores of the specified tensor~$N=\abs|\mathbb T|$ times, 
and then perform its convolution by the formula~\eqref{eq:convl} in the sparse format.
For the numerical experiments we take values $\{w_k\}$ as
i.i.d.\ random integers uniformly distributed on the interval~$[1,\,10]$
and take threshold equal to $M=\left\lfloor1/2\sum_kw_k \right\rfloor + 1$.

Technically, we can avoid passing vectors to the deriving functions,
but limit ourselves to an integer argument~$X$ equal to~$X=x[1]+x[2]N_{\text{big}}$,
where~$N_{\text{big}}$ 
is a sufficiently large integer ($N_{\text{big}}=2^{15}$).
Then in each derivative function we produce an ``unpacking''
$x[1]=X\mod N_{\text{big}}$
and
$x[2]=\left\lfloor X / N_{\text{big}}\right\rfloor$.

\paragraph{Bankruptcy}
This game has a function of $\max$ on the sum of the values,
so it is mathematically similar to the airport game,
but we have chosen a different way of constructing derivative functions.
Namely, we take
\begin{equation}
    f_0^j(x)=
    \left\{
    \begin{aligned}
        \{x[1]-c_{\text{add}}(i),\,x[2]\}&,& x[1]-c_{\text{add}}(i)&> 0,\\
        \None&,& \text{else}
    \end{aligned}
    \right.,
    \;
    f_1^j(x)=\{x[1],\,x[2]+1\}
    ,\;\; 1\leq j < d,
\end{equation}
where 
$c_{\text{add}}(i)=c[i]$ if $i<k$ and
$c_{\text{add}}(i)=c[i+1]$, else.
For the middle function, which is the last one ($l=d=N-1$), 
we have:
\begin{equation}
    f_0^d(x)=
    \left\{
    \begin{aligned}
        \bigl(x[1]-c_{\text{add}} - \max(0,\, x[1]-c_{\text{add}} - c[k])\bigr)p(x[2])&,& x[1]-a_{\text{add}}&> 0,\\
        0&,& \text{else}
    \end{aligned}
    \right.,
\end{equation}
where $c_{\text{add}}=c[d+1]$ if $k<d+1$, and $c_{\text{add}}=c[d]$ if $k=d+1$;
\begin{equation}
    f_1^d(x)=\bigl(x[1] - \max(0,\, x[1] - c[k])\bigr)p(x[2]+1).
\end{equation}

Note that in the case of this example, the initial functions~$f_1$ are given the value~$\{E,\,0\}$ as  input
instead of~$0$.
This is done to simplify the type of derivative functions.
Like in Weighted majority game, derivative
functions take as input a vector of~$2$ elements,
the first of which accumulates the values of~$\{c_j\}$, 
and the second element summarizes the number of 
players in a coalition.

In this example, we construct the cores of the specified tensor~$N=\abs|\mathbb T|$ times, 
and then perform its convolution by the formula~\eqref{eq:convl} in the sparse format.
For the numerical experiments we take values $\{c_k\}$ as
i.i.d.\ random integers uniformly distributed on the interval~$[1,\,10]$
and take the value of~$E$ equal to $E=1/2\sum_kc_k$.

\paragraph{One seller market game}
In this game, first player in selling some good, players $2,\,\ldots,\,\abs|\mathbb T|$ offer prices $\{a_2,\,a_3,\ldots,\,a_{\abs|\mathbb T|}\}$ for this good, $a_k\geq0$.
If the first player is in a coalition $\mathbb S$, the  value of this coalition
is equal to the maximum price offered by the members of this coalition
\begin{equation}
    \nu(\mathbb S)=\max_{k\in\mathbb S,\,k\neq 1}a_k.
\end{equation}
If there is no first player in the coalition, its price is zero: $\nu(\mathbb S)=0$.

For the first player we take
the following derivative functions
\begin{equation}
    f^1_0(x)=\None,\quad f^1_1(x)=0,
\end{equation}
and for the rest of the players:
\begin{equation}
    f^k_0(x)=x,\quad f^k_1(x)=\max(x,\,a_k).
\end{equation}

The problem under consideration is another 
example of such a tensor construction, 
the ranks of which depend on the order of the indices.
For an unsuccessful sequence, the TT-ranks of the resulting tensor can be large and
it is necessary to perform an SVD-step for reducing the TT-ranks.
The final ranks depend on specific values of~$\{a_k\}$.

To reduce the ranks, we 
we can change the sequence of players, as it will not affect the calculation of the sum~\eqref{eq:games_semi}.
Namely,
place the first player first, and sort the other players according to their pre-determined prices in descending order.
With this sorting we can
take the same derivative function for $k>1$ as in Airport game.

\todel{
simplify the derivative functions:
\begin{equation}
    f^k_0(x)=x,\quad
    f^k_1(x)=
    \left\{
        \begin{aligned}
            x&,   &x&>0\\
            a_k&, &x&=0,
        \end{aligned}
    \right.
    ,\quad k>1.
\end{equation}
Now let us do the following trick: we  break the second function~$f^k_1$ into two terms
\begin{equation}
    f^k_1(x)=
    f^k_{(1)}(x)+
    a_k\cdot f^k_{(2)}(x),
    \quad
    f^k_{(1)}(x)=
    \left\{
        \begin{aligned}
            x&,   &x&>0\\
            \None&, &&\text{else}
        \end{aligned}
    \right.
    ,
    \;\;
    f^k_{(2)}(x)=
    \left\{
        \begin{aligned}
            1&,   &x&=0\\
            \None&, &&\text{else}
        \end{aligned}
    \right..
\end{equation}
}

Thus, in this problem we come to the tensor network shown in Fig.~\ref{fig:game_contr} 
with matrices~$A$ of the same kind as the matrices in the airport problem:
\begin{figure}[!htb]
\centering
\resizebox{20em}{!}{%
\begin{tikzpicture}[scale=2.50]
    \def\maxd{6}
    \node[draw, shape=circle] (G1) at (0,-1) {$\mathcal G_1$};
    \draw [thick] (G1.north) -- node[right] {$2$} ++(0,+1.7);
 \foreach \x in {2,...,\maxd}
      {
          \node[draw, shape=circle] (v\x) at (\x-1,0) {$A_{\x}$};
          \node[draw, shape=circle] (G\x) at (\x-1,-1) {$\mathcal G_{\x}$};
          \draw [thick] (v\x) -- node[right] {$3$} ++(0,-0.8) --  (G\x);
          \draw [thick] (v\x.north) -- node[right] {$2$} ++(0,+0.7);
      }
      \foreach \x in {1,...,\fpeval{\maxd - 1}}
      {
      \draw [thick] (G\x) -- node[above] {$2$} ++(0.8,0) -- (G\fpeval{\x + 1});
      }
\end{tikzpicture}%
}%
\caption{Tensor network for building TT-tensor for the one seller market game with optimal players ordering for $\abs|\mathbb T|=6$.
 The  numbers near lines denote ranks (number of terms in sum).
}
\label{fig:game_contr}
\end{figure}
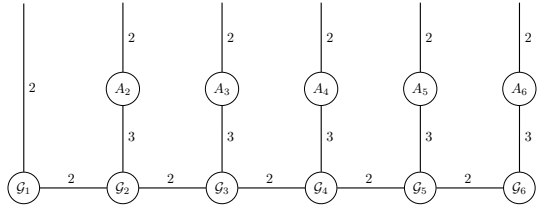
and matrices $A_k$ has the following elements
\begin{equation}
    A_k=
    \begin{pmatrix}
        1&0&0\\
        0&1&a_k
    \end{pmatrix},\quad
    k=\range2{\abs|\mathbb T|}.
\end{equation}

Technically, the convolution procedure can be omitted explicitly, taking it into account at the stage of tensor construction.

Note that since the cores~$\mathcal G$ of the resulting TT-tensor
are constructed using the same pattern as well as the cores of the tensor representing values of~$p$,
they can be built on the fly
rather than stored when calculating the sum~\eqref{eq:games_semi}.
Moreover, we can abandon the construction of the cores at all,
immediately producing all algebraic operations,
given the special (sparse) kind of kernels.
For the considered problem,
the sequence of matrix operations to calculate the sum~\eqref{eq:games_semi}
is reduced to the Algorithm~\ref{a:build_sum_in_game}.%
\begin{algorithm}[tbh!]%
    \caption{Algorithm for calculating the sum~\eqref{eq:games_semi}for one seller market game, based on obtained TT-cores}
    \label{a:build_sum_in_game}
    \begin{algorithmic}[1]
        \Require{Values ~$\{a_i\}$ sorted in descending order, number~$k$, weight function~$P$}
        \Ensure{Value of $\pi_k$}
        \State{$v\gets \{\underbrace{0,\,0,\,\ldots,\,0}_{2\abs|\mathbb T|}\}$ \AlgComm{Initialization}}
        \State{$v[0:4]\gets \{0,\,1\}\otimes\{1,\,0\}$ }
    \For{$i=1$ to $\abs|\mathbb T|-2$}
    \If{$i==k$}
    \For{$j=i$ to $0$ step $-1$}
    \State{$v[2j+1]\gets a_{i+1}\cdot v[2j]$}
    \State{$v[2j]\gets -v[2j]$}
    \EndFor
    \Else
    \For{$j=i$ to $0$ step $-1$}
    \State{$v[3+2j+1]\gets v[3+2j+1] + a_{i+1}\cdot v[2j] + v[2j+1]$}
    \EndFor
    \EndIf
    \EndFor
\State{$s\gets0$}
\For{$i=1$ to $\abs|\mathbb T|-2$}
\State{$s\gets s + v[2i]P(i+1)a_{\abs|\mathbb T|}$}
\State{$s\gets s + v[2i+1](P(i) + P(i+1))$}
\EndFor
\State{$s\gets s + v[2\abs|\mathbb T|-1]P(\abs|\mathbb T|-1)$}
\State{Return $s$}
\end{algorithmic}
\end{algorithm}
This algorithm works only for $1<k<\abs|\mathbb T|$ and it is not the most efficient.
However, it is shown to illustrate the possibility of applying the technique described in the article to build this kind of efficient iterative algorithms.
It is worth noting that the complexity of this algorithm is~$O(\abs|\mathbb T|^2)$.

\todel{
\subsubsection{Markov and non-Markov processes}
Our approach allows us to immediately write a discrete Markov process in terms of a tensor train.
Indeed, if a system can be in $N$ states at times $0,\,1\,\,\ldots,\,T$,
and we have a transition matrix $A\in\mathbb R^{N\times N}$
\begin{equation}
    A_{ij}=\mathbb P[s_n=i|s_{n-1}=j],
\end{equation}
where $s_k$ is the state at the time $k$.

\subsubsection{Dynamical programming}
\paragraph{Fibonacci number}
\paragraph{Number of terms in sum}
\paragraph{Robot}
}

\todel{
\subsubsection{Matrix calculus: Permanent}
Consider a task of calculating permanent of a matrix~$\{a_{ij}\}=A\in\mathbb C^{d\times d}$.
To solve this problem using the presented technique, let us construct two tensors in TT-format.
The first tensor~$\mathcal A$ will represent products of matrix~$A$ elements in the form
\begin{equation}
    \ttens A{\rangeidx id}=
    a_{i_1 1}
    a_{i_2 2}
    \cdots
    a_{i_d d}.
\end{equation}
This is rank-1 tensor and its cores~$\{\mathcal H_k\in\mathbb C^{1\times d\times1}\}_{k=1}^d$ are
\begin{equation}
    \tc H(1,i,1)
    =a_{ik},
    \quad
     i=\rangeo d.
\end{equation}

\begin{figure}[!h]
\centering
\resizebox{20em}{!}{%
\begin{tikzpicture}[scale=2.50]
 \foreach \x in {1,2,...,6}
      {
          \node[draw, shape=circle] (v\x) at (\x-1,0) {$a_{:,\,\x}$};
          \node[draw, shape=circle] (G\x) at (\x-1,-1) {$\mathcal G_{\x}$};
          \draw [thick] (v\x) -- node[right] {$6$} ++(0,-0.8) --  (G\x);
      }
\def\x{0}
\foreach \x/\r in {1/6,2/15,3/20,4/15,5/6}
      {
          \draw [dashed] (v\x) -- node[above] {$1$} ++(0.8,0) -- (v\fpeval{\x + 1});
          \draw [thick] (G\x) -- node[above] {$\r$} ++(0.8,0) -- (G\fpeval{\x + 1});
      }
\end{tikzpicture}%
}
\caption{Tensor network for calculating of the permanent, $d=6$.
 The  numbers denote ranks (number of terms in sum).
The dashed line shows that at this point the rank is~$1$, which means that there is no summation and the corresponding cores are independent.
}
\label{fig:permanent_net_2}
\end{figure}

The second tensor 
is an indicator tensor for such a set of indices,
in which all indices are different
\begin{equation}
    \ttens I{\rangeidx id}=
\left\{
\begin{aligned}
    1&,& &\text{ if all $\rangeidx id$ are different},\\
    0&,& &\text{ else}.
\end{aligned}
\right.
    \label{}
\end{equation}
The cores~$\mathcal G$ of this tensor are obtained using the following derivative functions
\begin{equation}
    f_i^k(x)=
    \left\{
    \begin{aligned}
        x+2^i&,& x\,\&\,2^i&=0,\\
        \text{None}&,& & \text{ else}
    \end{aligned}
    \right.,
    \label{}
\end{equation}
where the ampersand sign stands for bitwise AND for integers.
In this scheme
the middle-core is the last ($d$-th) core
and its derivative functions~$f^d$ are similar to those given:
\begin{equation}
    f_i^d(x)=
    \left\{
    \begin{aligned}
        1&,& x\,\&\,2^i&=0,\\
        \text{None}&,& & \text{ else}
    \end{aligned}
    \right..
    \label{}
\end{equation}
The $k$-th bit of the input~$x$ of the derivative functions 
corresponds to the $i$-th value of the set $\{1,\,2,\,\ldots,\,d\}$:
if this bit is zero, then the $i$-th value of the index has not occurred before,
and thus this value is valid for the current index.
The function sets this bit in~$x$ and the value~$x$ is passed on.
If the bit has already been set,
the derivative function returns None, since two identical indexes are forbidden.

The permanent value is equal to the convolution of tensor~$\mathcal I$
and tensor~$\mathcal A$.
Since the tensor~$\mathcal A$ is one-rank tensor,
we can look at the computation of the permanent as an contraction of the tensor~$\mathcal I$ with weights equal to the corresponding elements of the given matrix~$A$,
see Fig.~\ref{fig:permanent_net}.

To calculate the permanent numerically,
we do not use rounding;
all operations are performed in sparse format
using only the functions obtained by Algorithm~\ref{a:build_TT_func_f}.
The TT-ranks for some values of~$N$ are shown in the Table~\ref{tab:rans-perm}.
\begin{table}
    \centering
    \begin{tabular}{r|p{35em}}
        N &
        TT-ranks \\\hline\hline
        5 & 1---5---10---10---5---1 \\
        10 &
1---10---45---120---210---252---210---120---45---10---1 \\
        15 & 
1---15---105---455---1365---3003---5005---6435---6435---5005---3003---1365---455---105---15---1\\
    \end{tabular}
    \caption{TT-ranks for the problem of permanent calculation}
    \label{tab:rans-perm}
\end{table}
After calculating the function~$f_i^k$ we get a number~$x$ in the binary representation of which
there are exactly~$k$ bits equal to one.
From this one can concludes that the corresponding rank~$r_k$ is equal to $r_k=\frac{d!}{(d-k)!k!}$.
Using relation~\eqref{eq:convolve_w} we can obtain an upper estimate on the total number of operations for the convolution of the given tensor
as $n_{\text{conv}}= O(2^N N)$.
However, one can notice that on average half of the rows in the cores slices consist entirely of zeros, since index repetition is ``forbidden'':
$\mathcal G_k(i,\,j,\,:)=0$ at index~$i$ corresponds to the~$x$ with the $j$-th bit set.

Thus,
after the cores are built,
the number of operations (additions and multiplications) required to obtain the result at these ranks has asymotics
\begin{equation}
    2n_{\text{conv}}=
    O(2^N N).
    \label{}
\end{equation}
This asymptotic is better
than the one that
can be obtained from the well-known Ryser's formula for calculating the permanent
$$
P(A)=
(-1)^N
\sum_{\mathbb S\subseteq\{1,2,\ldots,N\}}
(-1)^{\abs|\mathbb S|}
\prod_{i=1}^N
\sum_{j\in \mathbb S}a_{ij}.
$$
When applied head-on, this formula requires~$O(2^{N-1}N^2)$ operations.
It is true that if one uses
Hamiltonian walk on $(N-1)$-cube (Gray code)
for a more optimal sequence of subset walks,
this formula will give asymptotic~$O(2^{N-1}N)$
which is only twice as good as ours~\cite[pp. 220--224]{CA1978}.
}

\subsubsection{Knapsack problem}
Consider a 
knapsack problem.
The formulation is the following:
we have~$n$ type of items, each item have weight~$w_i\geq0$ and value~$v_i\geq0$, $i=1,\,2,\,\ldots,\,n$.
The task is to solve optimization problem
\begin{equation}
    \maximize\limits_{\{\rangeidx xn\}}
    \sum_{i=1}^nv_ix_i,
    \quad 
    \text{s.t. }
    \sum_{i=1}^nw_ix_i \leq W,
\end{equation}
where $\{x_i\}_{i=1}^n$
are the unknown number of each item,
value~$W\in\mathbb R$ is the given maximum total weight.
By imposing different constraints on the unknowns~$x_i$,
we obtain different formulations of the problem.

\paragraph{0--1 knapsack problem}
First, consider a formulation in which it is assumed that $x_i\in\{0,\,1\}$ (so called \emph{0--1 knapsack problem}).
To solve this problem, we construct two tensors.
First tensor~$\mathcal V$ represent the total cost of the knapsack
\begin{equation}
    \mathcal V(\rangeidx xn)=
    \sum_{i=1}^dv_ix_i.
\end{equation}
Since the expression
on the right side of this definition
is linear,
this tensor can be represented in TT-format with TT-ranks of no more than~$2$ (see~\eqref{eq:ex:sum}).
Namely, the cores~$\{\mathcal H_k\}$ of the TT-decomposition of the tensor~$\mathcal V$
are
\begin{gather}
    \ts{H_1}0
    =\bigl(1,\, 0 \bigr);
    \quad
    \ts{H_k}0=
\begin{pmatrix}
1 & 0\\
0 & 1
\end{pmatrix},
\;\;2\leq k\leq d-1;
    \quad
    \ts{H_d}0=
\begin{pmatrix}
1\\
0
\end{pmatrix};\\
    \ts{H_1}1
=\bigl(1,\,v_1 \bigr);
    \quad
    \ts{H_k}1
\begin{pmatrix}
    1 & v_k\\
0 & 1
\end{pmatrix},
\;\;2\leq k\leq d-1;
    \quad
    \ts{H_d}1=
\begin{pmatrix}
    v_d\\
1
\end{pmatrix}.
\end{gather}
Here we use indices~$0$ and~$1$ for the middle indices of the TT-cores
so that they correspond to the physical sense of the task---the presence or absence of this item in the knapsack.

The second tensor~$\mathcal I$ is the indicator tensor of the condition in the knapsack problem:
\begin{equation}
    \ttens I{\rangeidx xn}
    =\left\{
    \begin{aligned}
        1&, \text{ if }\sum_{i=1}^nw_ix_i \leq W,\\
        0&, \text{ else}.
    \end{aligned}
    \right.
\end{equation}

We can build functional TT-decomposition of the tensor~$\mathcal I$ with the following set of functions
\begin{equation}
    f_i^k(x)=
    \left\{
    \begin{aligned}
       x+iw[k]&,& x+iw[k]& \leq W,\\
        \None&,& \text{else}
    \end{aligned}
    \right.,
    \quad i=0,\,1
    ;\;\;
    1\leq k\leq d.
\end{equation}

Note that the condition of not exceeding the weight is checked in each function, \ie, the conditions for partial sums are checked: $w_1x_1\leq W$, $w_1x_1+w_2x_2\leq W$, etc.
This does not affect the final result as $w_i\geq0$, but it allows us to reduce the ranks of the cores.
The TT-ranks in this problem are highly dependent on the specific weights of the knapsack elements:
whether they are integer, how large a range of values they have, etc.

The final answer to the knapsack problem is given by finding the maximum of the tensor,
which is the elementwise product of the constructed tensors
\begin{equation}
    \argmax_{\{\rangeidx xn\}}
    \ttens V{\rangeidx xn} \cdot \ttens I{\rangeidx xn}.
\end{equation}
The TT-cores of such a tensor are found as the Kronecker product of the slices of the multiplier cores (see~\cite{Oseledets2011}).
The operation of approximate finding the maximum value of the tensor in TT-format is implemented, 
for example, in the package \texttt{ttpy}.

\todel{
\paragraph{Bounded knapsack problem}
We can easily extend this approach to the problem of the \emph{bounded knapsack problem}.
This problem differs in the conditions on unknowns~$x_i$, now we let
\begin{equation}
    x_i=0,\,1,\,\ldots N_i,\;\;N_i\geq0,\;\;i=\rangeo n.
\end{equation}
In this formulation,
the functions~$f$ will remain the same as in~\eqref{eq:f_KNP},
and the cores of the tensor~$\mathcal V$ will have greater dimensionality: $\mathcal H_i\in\mathbb R^{r_{i-1}\times (N_i+1)\times r_i}$,
$r=\{1,\,2,\,2,\,\ldots,\,2,\,1\}$,
with the following explicit form
\begin{equation}
    \mathcal H_1(:,\,j,\,:)=\bigl(jv_1,\, 1 \bigr);
    \quad
    \mathcal H_k(:,\,j,\,:)=
\begin{pmatrix}
1 & 0\\
jv_i & 1
\end{pmatrix},
\;\;2\leq k\leq n-1;
    \quad
    \mathcal H_n(:,\,j,\,:)=
\begin{pmatrix}
1\\
jv_n
\end{pmatrix}.
\end{equation}
}
\paragraph{Multi-Dimensional bounded knapsack problem}
Another addition,
which can be implemented with slight modifications of the above scheme,
is related to the presence of several constraints.
Namely,
let there be several weights~$\{w_i^{(j)}\}$ associated with each item,
with separate conditions imposed on them
\begin{equation}
    \maximize\limits_{\{\rangeidx xn\}}
    \sum_{i=1}^nv_ix_i,
    \quad 
    \text{s.t. }
    \sum_{i=1}^nw_i^{(j)}x_i \leq W^{(j)},\;\;j=\rangeo M;\;\;
    x_i=0,\,1,\,\ldots N_i,
\end{equation}
where~$M$ is the length of condition vector.

To solve this problem, we generate~$M$ indicator tensors~$\{\mathcal I^{(j)}\}_{j=1}^M$,
one for each condition,
according to the algorithm above.
The we find element-wise product of all this tensors, and find (quasi-) maximum element in the resulting tensor
\begin{equation}
    \argmax_{\{\rangeidx xn\}}
    \mathcal V(\rangeidx xn) \cdot \prod_{j=1}^M\mathcal I^{(j)}(\rangeidx xn).
\end{equation}
The product of a large number of tensors~$\mathcal I^{(j)}$ can lead to a rapid growth of ranks,
to get around this we can round the tensor after each successive multiplication.

\subsubsection{Partition problem}
Consider the partition problem in the following formulation.
We have a multiset~$S$, $\abs|S|=d$ of positive integers
(multiset is a set with possibly repeating elements) and an integer~$n$.
The task is to partition the set~$S$ on~$n$ sets~$\{S_i\}_{i=1}^n$
such that the sum of elements in each set are equal:
\begin{equation}
    \sum_{a\in S_1}a=
    \sum_{a\in S_2}a=
\cdots
=
    \sum_{a\in S_n}a=
    \frac1n\sum_{a\in S}a.
\end{equation}

In order to use our approach to solve this problem,
we construct~$n$ indicator tensors,
each of which corresponds to one of the equalities in the expression above.
Namely, the $j$-th tensor is defined as
\begin{equation}
    \mathcal I_j(\rangeidx id)=
    \left\{
    \begin{aligned}
        1&, \text{ if }\sum_{k=1}^ds[k]\cdot \delta(i_k,\,j) =T,\\
        0&, \text{ else}
    \end{aligned}
    \right.,
\end{equation}
where $s[k]$ is the $k$-th element of the (ordered in some way) set~$S$,
$\delta$ is the Dirac delta function
and
$T\defval \frac1n \sum_{a\in S}a$.

The maximum indices value of this tensor is~$n$.
The index value $i_k=l$ means that the $k$-th element of the set~$S$ belongs to the set~$S_l$.
Thus we have one-to-one correspondence between indices set $\{\rangeidx id\}$ and a partition of the set~$S$.

Derivative functions for the $j$-th tensor are the following
\begin{equation}
    f_i^k(x)=
        \left\{
    \begin{aligned}
        x+s[k]&, \text{ if } i=j,\\
        x&, \text{ else}
    \end{aligned}
    \right.,
    \quad
    1\leq k < d,
    \;\;
    1\leq i\leq n,
\end{equation}
and the middle function which is the last one:
\begin{equation}
    f_i^d(x)=
        \left\{
    \begin{aligned}
        1&, &\text{ if } x+s[d]&=T \text{ and } i=j,\\
        1&, &\text{ if } x&=T \text{ and } i\neq j,\\
        0&, &\text{ else}
    \end{aligned}
    \right.,
    \quad
    1\leq i\leq n.
\end{equation}

Finally we construct indicator tensor~$\mathcal I$ as a Hadamard product of the built tensors
\begin{equation}
    \mathcal I = \mathcal I_1\circ\mathcal I_2\circ\cdots\circ\mathcal I_n.
\end{equation}
The value of this tensor~$\mathcal I$ is $1$
only for the indices corresponding to the desired problem statement.
We can find them by finding the maximal element of the indicator TT-tensor, see Section~\ref{seq:fnd_nnz}
in the Appendix.

\subsubsection{Eight queens puzzle}
Consider the classical problem of eight queens and its extensions.
In the classical variant it is necessary to place 8 queens on a usual $8\times8$ chessboard so that the queens do not beat each other.
In other words, no two queens stand on the same vertical, horizontal, or diagonal.
We will consider this problem on an $N\times N$ board with~$N$ queens.
To solve this problem by our techniques
we construct a tensor of dimension~$N$, each $k$-th index~$i_k$ of which  denote the position of the queen on the $k$-th vertical, so $1\leq i_k\leq N$.
The value of the tensor is an indicator function of the desired state:
$1$ if the location of the queens satisfies the condition
and zero otherwise.

The derivative functions~$f^i_k(x)$ in Python for such a tensor are shown in Fig.~\ref{fig:8_queen_code}.
The are the same for all~$k$, middle-core is the last core,
the above code also covers the function for the last core.

Let us briefly describe their work. The input is not a single number, but an array (it is allowed in our scheme) of zeros and ones.
The first~$N$ bits show the position of the previous (leftmost) queens horizontally. 
The bit corresponding to the current queen is added to them. 
The next~$N$ bits show those fields which are broken from bottom to top, and they are shifted forward.  
And finally, the last series of~$N$ bits shows those fields that are broken from top to bottom, and they are shifted back. 
If the position of the bit corresponding to the position of the current queen contains a one in at least one of the sets,
the producing function returns \None, since it means that the condition has already been violated.
\begin{figure}
\begin{subfigure}[b]{0.45\textwidth}
         \centering
\begin{lstlisting}[language=Python]
def f(x, i):
    last = i == N-1
    if (x[i]) > 0 or \
       (x[i + N]) > 0 or \
       (x[i + 2*N]) > 0:
        return None
 
    if last: # middle-core
        return 1
 
    # add our position 
    # to all sets
    x[i] = 1
    x[i + N] = 1
    x[i + 2*N] = 1
    
    # shift bottom-to-top set
    x[N:2*N-1] = x[N+1:2*N]
    x[2*N-1] = 0
 
    # shift top-to-bottom set
    x[2*N+1:3*N] = x[2*N:3*N-1]
    x[2*N] = 0
 
    return x
\end{lstlisting}
\caption{derivative functions in Python}
\label{fig:8_queen_code}
\end{subfigure}
\hfill
\begin{subfigure}[b]{0.45\textwidth}
         \centering
         \parbox{0.99\linewidth}{
             \def\iclchess#1{\includegraphics[width=0.75\linewidth]{chess_#1.pdf}\\}
            \iclchess9
            \iclchess{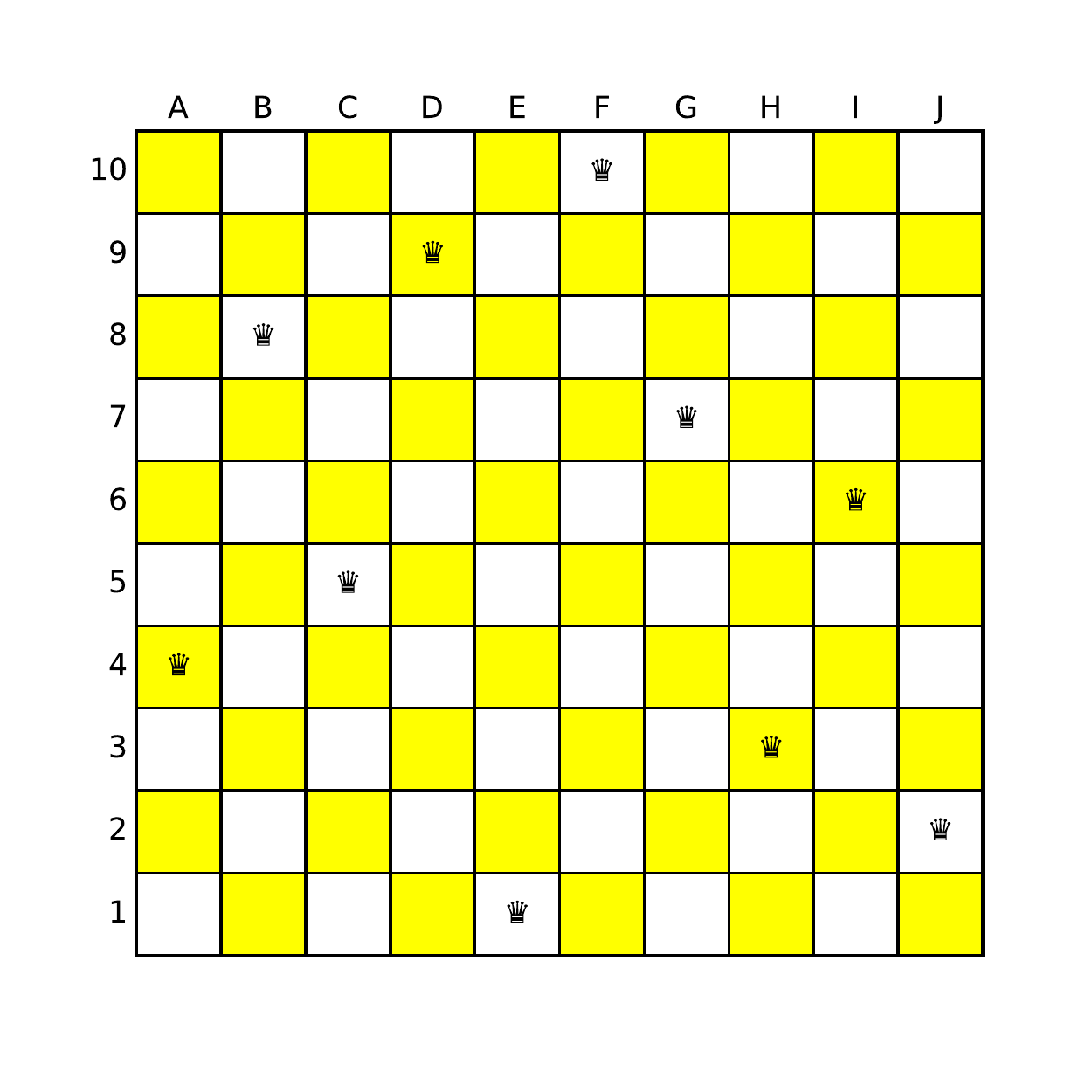}
         }
\caption{Positions for $N=9$ and $10$.}
\label{fig:8_queen_pos}
\end{subfigure}
\caption{Eight queens puzzle}
\label{fig:8_queen}
\end{figure}

After constructing the tensor,
we can find the desired position: this problem is reduced to finding a non-zero element in the tensor (in this case the value of~$1$), and its algorithm is described in the Appendix, Section~\ref{seq:fnd_nnz}.
The result is shown on Fig.~\ref{fig:8_queen_pos}.

Note that although the state that is passed to each derivative function is~$3N$ in length
and thus can potentially take~$2^{3N}$ combinations,
the real TT-rank is much smaller,
since only a small fraction of this set of combinations is admissible.
See Table~\ref{tab:8_queen_ranks} for the numerical values for the ranks
together with the error in calculating the total number of combinations
\begin{table}
    \def\err#1#2#3{\ensuremath{#1\pm#2\cdot{10}^{#3}}}
    \centering
    \footnotesize
    \begin{tabular}{p{9em}|c|p{7em}|}
        $N$ and truncation threshold $\epsilon$ & TT-ranks & \# of positions \\\hline\hline
    $8$, w/o truncation                       &  1---8---42---140---339---538---482---224---1 & 92 \\
    $8$,  $\epsilon=0$        & 1---8---36---62---74---62---36---8---1       & \err{92.0}2{-14} \\
    $8$,  $\epsilon=10^{-6}$  & 1---8---36---62---74---62---36---8---1  & $92.0\pm0.0$ \\
    $9$, w/o truncation                       &  1---9---56---234---726---1565---2153---1734---740---1 & 352\\
    $9$,  $\epsilon=0$        &  1---9---56---221---592---712---191---54---9---1  & \err{352.0}2{-13}\\
    $9$,  $\epsilon=10^{-6}$        &   1---9---54---172---246---246---172---54---9---1  & \err{352.0}4{-13}\\
    $10$, w/o truncation                       &  1---10---72---364---1393---3842---7289---8838---6426---2576---1 & 724\\
    $10$, $\epsilon=0$        &  1---10---72---339---914---1225---1820---391---72---10---1 & \err{724.0}5{-13} \\
    $10$, $\epsilon=10^{-6}$        & 1---10---72---284---526---606---526---284---72---10---1  & $724.0\pm0.0$ %
    \\
    \hline
    \end{tabular}
    \hbox{}
    \caption{
        TT-ranks and the number of positions calculated using the tensor depending on the size of the boar~$N$ and the truncation threshold~$\epsilon$.
    }
    \label{tab:8_queen_ranks}
\end{table}

Of course, solving the problem of finding permissible combinations in the described way is inefficient. But we can get the number of possible combinations even faster than finding one of them. This number is obtained by convolution of the tensor with vectors consisting of ones.
Besides, by constructing cores of decomposition, we can solve extended problems. For example, the complement problem, when some of the queens are already placed on the board, and we need to find the location of the others.
Or we can consider that queens are placed on the board randomly,
with a given probability distribution law,
and the problem is  to find the probability that such an arrangement will result in the position given by the puzzle rules.

\subsubsection{Combinatorial problems}
\paragraph{Sawtooth sequence}
Sawtooth sequence is a sequence of integer numbers $\{\rangeidx ad\}$ 
such that if $a_{i-1}<a_i$ then $a_i>a_{i+1}$,
and vice versa.

Suppose we have a set of arrays~$\{c_j\}$, 
and for a given set of indices $\rangeidx id$ we want to
construct an indicator tensor equal to 1 if the corresponding sequence
$\{c_1[i_1],\,c_2[i_2],\,\ldots,\,c_d[i_d]\}$
is sawtooth.

In this example, the derivative functions receive as input a sequence of two elements,
the first element contains the value of the previous member of the sequence,
and the second denotes the direction, ``up'' or ``down''.

\begin{equation}
    f_i^k(x) =  \left\{
    \begin{aligned}
        \{c_k[i],\,``up''\}&,\text{ if } c_k[i] < x[1] \text { and } x[2]=``down''\\
        \{c_k[i],\,``down''\}&,\text{ if } c_k[i] > x[1] \text { and } x[2]=``up''\\
        \None&, \text{ else}
    \end{aligned}
    \right..
\end{equation}
Middle function is the last one and
it is of the same form. 
We need to slightly 
alter this definition for the first function for $k=1$.

Such sequences are often found in game problems. 
We can combine this indicator tensor with other conditions. 
For example, we can find the number of possible tooth sequences with certain conditions through 
convolution~\eqref{eq:convl} of the resulting tensor with units.

\paragraph{Number of subsets}
Consider Problem \#10 from the Advanced problems chapter of the book~\cite{Andreescu2002-dc}:
Find the numbers of subsets of $\{\rangeo2000\}$,
the sum of whose elements is divisible by 5.

We can immediately construct an indicator tensor
with binary indices, which equals~$1$
if the given subset (of indices with value 1)
satisfies the condition of the problem.
Namely, let us take the following derivative functions
\begin{equation}
    f_i^k(x)=(x+ik)\mod 5,\quad 1\leq k < d,\;\;i=0,\,1,
\end{equation}
and the middle function which is the last one in this example as
\begin{equation}
    f_i^d(x)=
    \left\{
    \begin{aligned}
        1&,\text{ if }  (x+ik)\mod 5=0\\
        0&,\text{ else }\\
    \end{aligned}
    \right..
\end{equation}
However, for the specified number of elements of the sequence (2000),
it can be a time-consuming task to convolve such a tensor 
even in a sparse format. 
Therefore, let us try to solve this problem analytically,
using the explicit representation of the cores of this tensor.

Note that for the product of any five cores of this tensor, 
starting from the number that gives~$1$ when divided by~$5$, 
it is true
\def\oneterm#1{\left(\sum_{i_{#1}=1}^2\mathcal G_{#1}(:,\,i_{#1},\,:)\right)}
\begin{multline}
    \oneterm{5n+1}
    \oneterm{5n+2}
    \oneterm{5n+3}\times\\
    \times
    \oneterm{5n+4}
    \oneterm{5n+5}
    =
    \begin{pmatrix}
    8&6&6&6&6\\
    6&8&6&6&6\\
    6&6&8&6&6\\
    6&6&6&8&6\\
    6&6&6&6&8\\
    \end{pmatrix}.
\end{multline}
Thus, to solve the problem it is necessary to power this matrix to 2000/5=400.
To do this, let us write this symmetric matrix in the diagonal form:
\begin{equation}
    \begin{pmatrix}
    8&6&6&6&6\\
    6&8&6&6&6\\
    6&6&8&6&6\\
    6&6&6&8&6\\
    6&6&6&6&8\\
    \end{pmatrix}
    =A 
    \begin{pmatrix}
    32&0&0&0&0\\
    0&2&0&0&0\\
    0&0&2&0&0\\
    0&0&0&2&0\\
    0&0&0&0&2\\
    \end{pmatrix}
A^T.
\label{eq:mat86}
\end{equation}
By direct calculations we find that for the matrix A and any $q\,,s\in\mathbb R$ we have
\begin{equation}
 B:=A 
    \begin{pmatrix}
    q&0&0&0&0\\
    0&s&0&0&0\\
    0&0&s&0&0\\
    0&0&0&s&0\\
    0&0&0&0&s\\
    \end{pmatrix}
A^T=
\frac15
\left(
\begin{array}{ccccc}
 q+4 s & q-s & q-s & q-s & q-s \\
 q-s & q+4 s & q-s & q-s & q-s \\
 q-s & q-s & q+4 s & q-s & q-s \\
 q-s & q-s & q-s & q+4 s & q-s \\
 q-s & q-s & q-s & q-s & q+4 s \\
\end{array}
\right).
\end{equation}
Given that the first and last cores
of the TT-disposition are vectors, not matrices,
we need only the first element of the first row~$B[1,\,1]=1/5(q+4s)$.
Using the diagonal form of the matrix~\eqref{eq:mat86}
we can immediately power it to degree 400 by power its eigenvalues to this degree and 
obtaining~$q=32^{400}$ and $s=2^{400}$.

Thus, the final answer to the problem: $1/5(32^{400}+2^{400})$ number of subsets.

\subsubsection{SAT problem}
Consider the standard 
Boolean satisfiability problem (SATisfiability)
in conjunctive normal form (CNF-SAT).
Given~$d$ Boolean variables~$\left\{ x_i \right\}_{i=1}^d$,
taking the value ``True'' or ``False''.
From these variables we form~$m$ logical expressions~$\left\{ A_i \right\}_{i=1}^m$ containing a set of some variables~$\{x_i\}$ or their negations (which we denote by the symbol~$\neg$)
united by logical OR (symbol $\vee$), for example:
\begin{equation}
    A_1=x_1 \vee \neg x_3 \vee x_5,
    \quad
    A_2=\neg x_1 \vee \neg x_2 \vee x_4 \vee  x_5 \vee  x_{10},
    \quad \text{etc.}
    \label{}
\end{equation}
The problem is
to determine
if there is such
a set of variables~$x$
that expressions~$A_i$, combined with logical AND (symbol $\wedge$), give a logical ``True'':
\begin{equation}
    A_1\wedge A_2 \wedge \ldots  \wedge A_m=\text{True}.
    \label{}
\end{equation}

Let us use our method to construct a tensor with~$n$ indices taking the values~0 and~1,
corresponding to logical ``True'' and ``False'',
which is equal to~$1$ if the latter equality is satisfied, and zero otherwise.
Each index corresponds to a different variable~$x_i$.

In this problem, we construct~$m$ separate indicator tensors for each value~$A_j$,
which we then multiply elementwise,
which corresponds to the logical AND operator.
The indices of these tensors are binary and correspond to the \texttt{True} or \texttt{False}
value of the corresponding variable.
The derivative functions for the $j$-th tensor are the following:
\begin{equation}
    f_{\text{True}}^k(x) =    \left\{
    \begin{aligned}
        x&,\text{ if the variable $x_k$ is not part of the condition $A_j$}\\
        x&,\text{ if the variable $x_k$ is included in the condition~$A_j$ with the negation}\\
        1&, \text{ else}
    \end{aligned}
    \right.,
\end{equation}
\begin{equation}
    f_{\text{False}}^k(x) =    \left\{
    \begin{aligned}
        x&,\text{ if the variable $x_k$ is not part of the condition $A_j$}\\
        x&,\text{ if the variable $x_k$ is included in the condition~$A_j$ without the negation}\\
        1&, \text{ else}
    \end{aligned}
    \right..
\end{equation}
The function that corresponds to the last numbered variable 
($x_{10}$ for $A_2$ in the above example)
returns \None{} if its argument $x=0$
thus zeroing out the tensor:
a zero value of~$x$ means 
that no member of the condition~$A_j$ took the value \texttt{True}.

In this problem,
the final rank of the constructed tensor depends strongly on the order of the indices,
although it should be noted that many SAT algorithms also have heuristics on the sequence of passes over the variables.

\subsection{Rank optimization}
The algorithm presented in the main Theorem~\ref{th:main}
does not always give optimal ranks.

Below in this section
we present
additional steps
that reduce the ranks. %
Let's now look more closely at the method of rank reduction described in section~\ref{sec:rank}.

In this method, we combine different, but close, values of the derivative functions.
Let the values of the derived functions be real
and we are given the parameter~$\epsilon$ of ``smallness''.
We start by sorting all possible set values of each of the sets $R[i]$, $i=\rangeo{d-1}$ in Algorithm~\ref{a:build_TT_func_f} in ascending order.
We then partition all the values into such non-overlapping sets whose difference between a larger and a smaller element does not exceed~$\epsilon$.
Namely, we put
\begin{equation}
    S[i][1] = \{x\in R[i]\colon R[i][1]\leq x < R[i][1]+\epsilon  \}
\end{equation}
and then sequentially define
\begin{equation}
    \def\minS{\min\Bigl\{R[i]\setminus\bigcup_{j=1}^{k-1}S[i][j] \Bigr\}}
    S[i][k] = \Bigl\{x\in R[i]\colon \minS\leq x < \minS+\epsilon  \Bigr\}
    \label{}
\end{equation}
until all elements in the initial set~$R[i]$ are exhausted.
If additionally a maximum rank is given
and the number of elements in~$S[i]$ exceeds it,
we combine the sets further,
trying to do so as uniformly as possible.

For each~$i$ we update~$R[i]$ after defining all sets~$S[i]$
at each step:
\begin{equation}
    R[i][k]\gets\text{average}\,(S[i][k]),
    \quad
    k=\rangeo \abs|S[i]|,
\end{equation}
where the \texttt{average} function can be chosen in different ways, 
reasonable choices for it are the arithmetic mean 
and the average between the maximum and minimum elements.
Finally, the change we make in the algorithm is to replace the \texttt{index\_of} function 
with the following \texttt{set\_index\_of} function
\begin{equation*}
    z = \text{set\_index\_of}(y,\,A)
   \; \Longleftrightarrow\;
   y\in A[z].
\end{equation*}
As the second argument, we pass the sets~$S$ to this function instead of~$R$. 
Technically, we use the \texttt{searchsorted} function from the \texttt{numpy} package, 
and work with interval bounds rather than looking in sets as stated in the definition.

\subsubsection{Decrease the number of outputs}
\def\S_#1{S[#1]}%
Assume that the middle-index is the last index: $l=d$.
This avoids duplicating operations.

If the number of possibles vales of the tensor is small, \ie{} the length of the array~$R[d]$ is small,
we can perform the following trick.
Consider the images~$(f_{i_d}^{(d)})^{-1}$ of the $d^{\text{th}}$ set of derivative functions
\begin{equation}
    (f_{i_d}^{(d)})^{-1}(a)
\defval
\{x\colon f_{i_d}^{(d)}(x)=a\},
\quad
i_d=\rangeo{n_d}
.
\end{equation}
The idea of reducing the rank of the TT-decomposition
is that the values from this set are indistinguishable for the function~$f_{i_d}^{(d)}$.
Thus, if there exists a set indistinguishable for all
set of the last derivative functions,
then its elements can be encoded by a single basis vector.
Namely,
consider the following array (ordered set) of sets
\begin{equation}
    \S_{d-1}
\defval
\left\{ s=\bigcap_{i=1}^{n_d}(f_i^{(d)})^{-1}(a_i)\colon  \{a_1,\,a_2,\,\ldots,\,a_{n_d}\}\in (R[d])^{\times n_d},\;\;s\neq\varnothing\right\}.
\end{equation}
The array~$\S_{d-1}$ contains all nonempty sets, that are indistinguishable for functions~$f_i^{(d)}$ for any value of~$i=\rangeo{n_d}$.
The order in which these sets are included in~$\S_{d-1}$ is unimportant, but it is fixed.
The number of all this sets is not greater than the number of outputs in $R[d-1]$: $\text{len}(\S_{d-1})\leq \text{len}(R[d-1])$
as each element of~$R[d-1]$ belongs exactly to one set in~$\S_{d-1}$.
Due to this, a decrease in ranks is achieved. 

This procedure is repeated sequentially for each output. At the~$k^{\text{th}}$ step we find an array of sets
\begin{equation}
\S_{k-1}
\defval
\left\{ s=\bigcap_{i=1}^{n_k}(f_i^{(k)})^{-1}(A_i)\colon  \{A_1,\,A_2,\,\ldots,\,A_{n_k}\}\in (\S_k)^{\times n_d},\;\;s\neq\varnothing\right\},
\end{equation}
where $\{A_i\}$ are sets and we define the image on a set as
\begin{equation}
    (f_{i_k}^{(k)})^{-1}(A)
\defval
\{x\colon f_{i_k}^{(k)}(x)\in A\}.
\end{equation}
After all sets~$\S_k$ ($k=\rangeo{d-1}$) are found, we construct integer-valued functions
by analogy with the functions~\eqref{eq:i_func_simple}
\begin{equation}
    \hat f^{(k)}_i(x)
    =
    \text{set\_index\_of}\left(f_i^{(k)}\bigl(
    \hat x\bigr),\,\S_i
\right),
    \quad
    \hat x\in \S_{k-1}[x]
,
\;\;
     q=\rangeo d.
    \label{eq:i_func_cmplx}
\end{equation}
We let $\S_d=\{\{x\}\colon x \in R[d]\}$.
The definition~\eqref{eq:i_func_cmplx} is correct
because the value of the functions~$f^{(i)}$ does not depend on the choice of a particular element~$\hat x$ in the set~$\S_{i-1}[x]$.
\subsubsection{Rounding}
The cores that are obtained using Algorithm~\ref{a:build_TT_func_G} have one important property---they are ,,almost'' orthogonal in the sense that
\begin{equation}
    \sum_{i}\sum_\alpha
    \mathcal G_k(\alpha,\,i,\,l)
    \mathcal G_k(\alpha,\,i,\,m)
    =\lambda_l \delta_{lm}
    ,\quad
    k=\rangeo{d-1}.
    \label{eq:Q_orth}
\end{equation}
This is true, because matrices~$Q_k(i)=\mathcal G_k(:,\,i,\,:)$ 
consist of rows either identically equal to the zero vector or the basis vector~$e(j)^T$ for some~$j$.
The values of the coefficients~$\lambda_l$ have the physical meaning
of the number of occurrences of the value~$l$ as the value of the function~$f^{(k)}$.
This number
is equal to the number of occurrences of row vector~$e(l)^T$ in all matrices~$Q_k(i)$.
Thus $\lambda_l>0$.
The consequence of relation~\eqref{eq:Q_orth} is the following theorem which is a modified Lemma~3.1 from~\cite{Oseledets2011}:
\def\thetheorema{C.1}
\begin{theorema}
    Let~$Q_k(i)=\mathcal G_k(:,\,i,\,:)$, $Q_k(i)\in\mathbb R^{r_{k-1}\times r_k}$, $r_0=1$, be indexed matrices with tensors $\mathcal G_k$ satisfying~\eqref{eq:Q_orth}.
    Let the matrix~$Z$ be defined as follows
    \begin{multline}
        Z(\overline{i_1i_2\ldots i_k},\,l)\defval
        Q_1(i_1)Q_2(i_2)\cdots Q_k(i_k)=\\=
        \sum_{\alpha_1,\,\ldots,\,\alpha_{k-1}}
        \mathcal G_1(1,\,i_1,\,\alpha_1)
        \mathcal G_2(\alpha_1,\,i_2,\,\alpha_2)
        \cdots
        \mathcal G_k(\alpha_{k-1},\,i_k,\,l).
    \end{multline}
Then the matrix~$Z$ satisfy the following
orthogonality condition
\begin{equation}
    (Z^TZ)(l,\,m)=
    \sum_{\overline{i_1i_2\ldots i_k}}
        Z(\overline{i_1i_2\ldots i_k},\,l)
        Z(\overline{i_1i_2\ldots i_k},\,m)
        =\Lambda_l\delta_{lm}
\end{equation}
with natural $\Lambda_l\in\mathbb N$.

\end{theorema}

Thus, when rounding the tensor with the algorithm  described in~\cite{Oseledets2011},
we can start with the second step of this algorithm,
skipping the orthogonalization step.
In the case of setting the accuracy~$\epsilon$
the threshold for discarding the singular numbers  must take into account the values~$\Lambda_l$.

In the case where we need an exact representation of the given tensor in TT-format,
but with the smallest possible ranks,
the following consequence helps us.

\def\theconsequence{C.1}
\begin{consequence}
The 
cores, obtained by
Algorithm~\ref{a:build_TT_func_G},
have the optimal ranks if the unfolding of the last core is of full-column rank.
\end{consequence}

\subsection{Finding non-zero element in a TT-tensor}\label{seq:fnd_nnz}
Consider the case when the TT-tensor is an indicator tensor of some subset of its index values,
\ie{} it is equal to one (or an arbitrary values greater than zero) on a small number of combinations of its index values and to zero on the other combinations.
The problem is to find at least one combination where the tensor value is greater than zero.

The solution to this problem is based on the algorithmic simplicity of multidimensional summation of the tensor defined in the TT format.
Let $\mathcal  I$ be the tensor of interest.
Consider the sum of values of~$\mathcal  I$ on all variables except the first one
\begin{equation}
    v_1(i)=
    \sum_{i_2=1}^{n_2} \cdots \sum_{i_d=1}^{n_d}\ttens I{i,\,i_2,\,\ldots,\,i_d}.
    \label{}
\end{equation}
Knowing~$v_1$, we can find out the value of $\hat \imath_1$ of the first index of the desired combination:
\begin{equation}
    \hat \imath_1 = \argmax_iv_1(i) > 0.
    \label{}
\end{equation}
Indeed, if $\hat \imath_1$ is not a part of any desired combination,
then for any values of other indices the value of the tensor is zero: $\ttens I{\hat \imath_1,\,i_2,\,\ldots,\,i_d}=0$.
But this contradicts the fact that the value of the sum over these variables is greater than zero.

Then we sequentially find the indices of the desired combination.
For the second index:
\begin{equation}
    v_2(i)=
    \sum_{i_3=1}^{n_2} \cdots \sum_{i_d=1}^{n_d}\ttens I{\hat \imath_1,\,i,\,i_3,\,\ldots,\,i_d},
    \label{}
\end{equation}
\begin{equation}
    \hat \imath_2 = \argmax_iv_2(i)
    \label{}
\end{equation}
ans so on.
This sequence of steps is summarized in Algorithm~\ref{a:ome_max_val_TT}.
\begin{algorithm}[tbh!]%
    \caption{Algorithm for calculating indices set of the non-zero value of a TT-tensor}
    \label{a:ome_max_val_TT}
    \begin{algorithmic}[1]
        \Require{Cores $\{\mathcal G_i\}_{i=1}^d$ of the TT-decomposition of the tensor $\mathcal I$ that takes non-negative values}
        \Ensure{Set of indices $\{\rangeidx {\hat \imath}d\}$ s.t. $\ttens I{\rangeidx {\hat \imath}d}>0$}

        \For{$k=1$ to $d$}
        \State{$G_k \gets \sum_{i=1}^{n_k}\ts {G_k}i$}
        \EndFor
        \For{$k=1$ to $d$}
        \State{$v_k(i) \gets \mathcal G_1(1,\,\hat \imath_1,\,:)\ts{G_2}{\hat \imath_2}\cdots \ts{G_k}{\hat \imath_{k-1}}\ts{G_k}{i}  G_{k+1}\cdots G_d $}
        \State{$\hat \imath_k \gets  \argmax\limits_iv_k(i)$}
        \EndFor
\State{Return $\{\rangeidx {\hat \imath}d\}$}
\end{algorithmic}
\end{algorithm}

\end{document}